\documentclass[10]{amsart}
\usepackage{url}
\usepackage{latexsym}
\usepackage{lscape} 
\usepackage{pdflscape}
\usepackage{float}
\usepackage{amssymb}
\usepackage{mathtools}
\usepackage{stmaryrd}%forthelllbrak

\usepackage[colorlinks=true,linkcolor=black,citecolor=black]{hyperref}
\usepackage[nameinlink,capitalise]{cleveref}
\usepackage{nicematrix}

\newtheorem{theorem}[equation]{Theorem}

\newtheorem{example}[equation]{Example}
\newtheorem{lemma}[equation]{Lemma}
\newtheorem{corollary}[equation]{Corollary}
\newtheorem{proposition}[equation]{Proposition}

\theoremstyle{remark}
\newtheorem{remark}[equation]{Remark}

\theoremstyle{definition}
\newtheorem{definition}[equation]{Definition}
\newtheorem{convention}[equation]{Convention}
\newtheorem{convention*}{Convention}

\numberwithin{equation}{section}

\newcommand{\R}{\mathbb{R}}           % Use for real numbers.
\newcommand{\C}{\mathbb{C}}           % Use for complex numbers.
\newcommand{\Z}{\mathbb{Z}}           % Use for integers.
\newcommand{\Q}{\mathbb{Q}}           % Use for rationals not quaternions
\newcommand{\End}{\text{End}}
 
\newcommand{\Hom}{\mathrm{Hom}}
\newcommand{\id}{\mathrm{id}}

\newcommand{\Ad}{\text{Ad}}
\newcommand{\ad}{\text{ad}}

\newcommand{\vol}{{\mathrm{vol}}}

\newcommand{\tr}{\ensuremath{\mathrm{tr}}}
\newcommand{\diag}{\ensuremath{\mathrm{diag}}}
\newcommand{\bg}{\ensuremath{\boldsymbol{\gamma}}}

\newcommand{\x}{\times}
\newcommand{\op}{\oplus}
\newcommand{\ox}{\otimes}
\newcommand{\w}{\wedge}
\newcommand{\txs}{\textsubscript}

\newcommand{\fg}{{\mathfrak g}}

\newcommand{\fh}{{\mathfrak h}}
\newcommand{\fk}{{\mathfrak k}}

\newcommand{\fp}{{\mathfrak p}}

\newcommand{\fq}{{\mathfrak q}}
\newcommand{\ft}{{\mathfrak t}}
\newcommand{\fu}{{\mathfrak u}}

\newcommand{\fz}{{\mathfrak z}}

\newcommand{\fsu}{\ensuremath{\mathfrak{su}}}

\newcommand{\fso}{\ensuremath{\mathfrak{so}}}

\newcommand{\Si}{\ensuremath{\mathrm{S}}}
\newcommand{\GL}{\ensuremath{\mathrm{GL}}}

\newcommand{\SU}{\ensuremath{\mathrm{SU}}}
\newcommand{\U}{\ensuremath{\mathrm{U}}}
\newcommand{\SO}{\ensuremath{\mathrm{SO}}}
\newcommand{\Sp}{\ensuremath{\mathrm{Sp}}}
\newcommand{\Spin}{\ensuremath{\mathrm{Spin}}}

\newcommand{\sprep}{\ensuremath{\mathrm{sp\,}}}
\newcommand{\Quad}{\mathcal{Q}}
\newcommand{\Div}{\ensuremath{\mathrm{Div}\,}}
\newcommand{\Hol}{\ensuremath{\mathrm{Hol}\,}}
\newcommand{\Herm}{\ensuremath{\mathrm{Herm}\,}}

\newcommand{\scal}{\ensuremath{\mathrm{scal\,}}}
\newcommand{\mono}[1][n]{c_{#1}}

\newcommand{\ga}{\alpha}

\newcommand{\gve}{\ensuremath{\varepsilon}}

\newcommand{\be}{\begin{equation}}
\newcommand{\beu}{\begin{equation*}}

\begin{document}

\baselineskip=18pt
\sloppy
\title[Equivariant Seiberg--Witten equations on Hermitian symmetric spaces]
%{Spectrum of pseudo-riemannian locally symmetric spaces\\ and admissibility of spherical representations}
{Equivariant Seiberg--Witten equations\\
on Hermitian symmetric spaces}
\author{Rafailia-Persefoni Tsiavou}
\thanks{{\em Mathematics Subject Classification MSC2020}. Primary: 53C27; Secondary 53C35, 57R57, 22E46}
\keywords{\em Seiberg--Witten theory, self-duality, Dirac operators, Hermitian symmetric spaces}
\address{Aristotle University of Thessaloniki, Department of Mathematics, Greece}
\email{rdtsiavou@math.auth.gr}
\maketitle

\begin{abstract}
We formulate and solve a $G$-equivariant analogue of the Seiberg--Witten equations on any irreducible Hermitian symmetric space $G/K$. We use a canonical $K$-invariant form to define a twisted Hodge operator generalising self-duality to higher dimensions. The resulting equivariant Seiberg--Witten equations arise, for a compact $G/K$, as conditions for the absolute minima of an energy functional. Their linear part %is governed by representation theory of $G$
has a representation theoretic interpretation: the same Weyl group data control the highest-weight realisation of the positive Dirac kernel in the compact case and its discrete-series realisation on the non-compact dual. This space of harmonic spinors collapses almost entirely by the quadratic conditions of the monopole equation. In the compact case, the equivariant moduli set, when non-empty, is a single gauge class, identifying with the $G$-invariant part of Bradlow's trichotomy for vortices, while in the non-compact case, every irreducible solution sits outside the discrete series.

\end{abstract}
\setcounter{tocdepth}{1}
\tableofcontents

\section{Introduction}\label{Intro}

The Seiberg--Witten equations emerged from an effort to understand how
particles interact when the forces between them become too strong for
standard approximation methods. In their study of a four-dimensional
supersymmetric gauge theory, Seiberg and Witten showed that changing the
description of the theory could make this difficult problem tractable
\cite{SW94a,SW94b}. In certain regimes, magnetic monopoles become the natural
objects to consider, and interactions that appeared strong in the original
description become weak in the new one. This insight allowed them to
determine the low-energy behaviour of the theory and, after a suitable
modification, to explain how monopole condensation can confine electric
charges. Witten recognised that these ideas could also be used to study
the geometry and topology of four-dimensional manifolds \cite{W94}. The
resulting Seiberg--Witten equations couple a connection on a line bundle to
a spinor field representing the monopole. Compared with the instanton
equations underlying Donaldson theory, this simpler system offered a more
accessible way to investigate the smooth structure of four-dimensional manifolds.
\medskip

Therefore, let $M$ be a closed oriented Riemannian four-dimensional manifold equipped
with a $\operatorname{Spin}^{c}$ structure $\mathfrak s$. Denote by
$S^+_{\mathfrak s}$ the positive spinor bundle and by
$L=\det S^+_{\mathfrak s}$ its determinant line bundle. For a unitary
connection $A$ on $L$ and a spinor
$\psi\in\Gamma(S^+_{\mathfrak s})$, the $\zeta$-perturbed Seiberg--Witten equations
are
\begin{equation}
  D_A^+\psi=0,
  \qquad
  F_A^+ + \zeta=-\frac12\sigma(\psi).
  \label{eq:classical-sw-intro}
\end{equation}
Here $D_A^+$ is the Dirac operator (acting on positive spinors), $F_A^+$ is the self-dual part of the
curvature, $\zeta$ is a self-dual perturbation, and $\sigma$ is the quadratic
map from positive spinors to self-dual two-forms. The variational origin of the unperturbed equations is expressed by the
energy functional
\begin{equation}
 \mathcal E(\psi,A)=
 \int_M\left(
   |\nabla_A\psi|^2+|F_A|^2
   +\frac{\operatorname{scal}}4|\psi|^2
   +\frac1{16}|\psi|^4
 \right)\operatorname{vol}_M .
 \label{eq:sw-energy-intro}
\end{equation}
The Lichnerowicz formula
\begin{equation}
 D_A^2=\nabla_A^*\nabla_A
       +\frac{\operatorname{scal}}4
       +\frac12\gamma(F_A^+)
 \label{eq:lichnerowicz-intro}
\end{equation}
relates the kinetic energy of the spinor to the square of the Dirac
operator and to the action of the self-dual curvature. Here, $\gamma$ is the Clifford multiplication and $\scal$ the scalar curvature of $M$. Combining this
identity with the four-dimensional decomposition
$F_A=F_A^++F_A^-$ and the norm identity for $\sigma(\psi)$ rewrites
\eqref{eq:sw-energy-intro}, up to a topological term, as a sum of
non-negative squares proportional to
\[
  |D_A^+\psi|^2
  \qquad\text{and}\qquad
  \left|F_A^++\frac12\sigma(\psi)\right|^2.
\]
Consequently, the solutions of the unperturbed equations, when they exist, are absolute minima; a self-dual perturbation $\zeta$ shifts the second square and
gives \eqref{eq:classical-sw-intro}. This energy identity explains why the
Dirac and curvature equations occur together and is the source of many of
the compactness and vanishing arguments in the theory. Solutions modulo the
abelian gauge group form the Seiberg--Witten moduli space and define the
Seiberg--Witten invariants, which retain deep information about smooth
four-dimensional manifolds while replacing the non-abelian instanton equation by an
abelian connection and a single spinor. On K\"ahler surfaces the same
variational structure also connects the equations with complex geometry and
vortex equations \cite{Br90,Ni00}.
\smallskip

The form of \eqref{eq:classical-sw-intro} is nevertheless special to dimension
four. The Hodge operator preserves two-forms only in four dimensions and
splits them into self-dual and anti-self-dual parts. At the same time,
Clifford multiplication identifies the self-dual two-forms with the traceless
Hermitian endomorphisms of the positive spinor bundle. These two facts allow
$F_A^+$ and the quadratic expression $\sigma(\psi)$ to take values in the
same bundle. In higher dimensions there is no such canonical coincidence.
Any generalisation must therefore provide both a replacement for self-duality
and a geometrically natural curvature component to which the spinor bilinear
can be equated.
\medskip

The starting point of this paper is that symmetry can supply the missing
structure. We consider manifolds admitting a transitive action of a
semisimple Lie group $G$, and hence presented as homogeneous spaces $G/H$,
with the aim of formulating and studying a $G$-equivariant analogue of the
Seiberg--Witten equations. This is also natural from the physical viewpoint.
The most familiar idealised spacetimes are homogeneous: Minkowski, de Sitter,
and anti-de Sitter spacetimes are determined by large transitive groups of
isometries. Their symmetry makes them basic backgrounds for field theory and
gravity. De Sitter and anti-de Sitter spaces are Lorentzian rather than
Riemannian, so they are motivations rather than direct examples of the
equations considered here; nevertheless, they illustrate why it is useful to
understand gauge fields and spinors on homogeneous backgrounds. On $G/H$,
bundles, connections, curvature, and differential operators can be described
by representations of the isotropy group $H$. Invariant connections are
controlled by finite-dimensional algebraic data \cite{Wan58}, and a Dirac
operator constructed from invariant data intertwines the action of $G$.
\medskip

The class on which this programme becomes particularly explicit is that of
irreducible Hermitian symmetric spaces $G/K$, of compact or non-compact type.
Such a space carries a $G$-invariant complex structure and K\"ahler form
$\omega$. At the base point, the complex structure is encoded by a
distinguished central element $Z_0\in\mathfrak k$ through
$J=\operatorname{ad}(Z_0)$. From $\omega$ one obtains a canonical
$K$-invariant form and a twisted Hodge operator; its positive eigenspace gives
the notion of self-duality used in this paper in every even dimension. The
homogeneous spinor bundle, the invariant curvature projection, and the
quadratic monopole map can consequently be written in explicit Lie-theoretic
terms. Hermitian symmetric spaces therefore provide all the ingredients
needed to retain one essential feature of the classical theory while suitably adapting the other:
%retain the two essential features of the classical theory: 
a curvature equation obtained from a self-dual component of the curvature and a $G$-equivariant Dirac equation.
\medskip

The first step is to identify the twists for which the spin representation
lifts from the isotropy algebra to the isotropy group. We establish a
correspondence between homogeneous $\operatorname{Spin}^{c}$ structures and
one-dimensional \emph{admissible twists} $E_\nu$ (\cref{1-1cor}). Such twists exist on every irreducible Hermitian symmetric space and are parametrised by
$\nu=1+2\kappa/N$, where $\kappa$ is an integer and $N$ is the dual Coxeter number; we write $E_\nu=E_{\nu\rho_{\mathfrak p}}$, where $\rho_{\mathfrak p}$ is half the sum of positive roots in $\fp$ and $\fg=\fk\op\fp$ the corresponding Lie algebra decomposition for $G/K$. The one-dimensionality of $E_\nu$ is important:
the auxiliary bundle remains a line bundle, exactly as required in the
Seiberg--Witten setting. If $\mathcal A^G$ denotes the invariant connections
and $\psi=\sum_k\psi^{(2k)}$ is the decomposition of a positive spinor by
exterior degree, then for every $A\in\mathcal A^G$ the corresponding Dirac
operator is the classical $D_0^+$ defined on symmetric spaces and is independent of $A$.
\medskip

The invariant K\"ahler form also determines a distinguished
$K$-invariant form $\Phi$ and a twisted Hodge operator $*_{\Phi}$. This provides a generalised notion of self-duality, which recovers the standard one in the four-dimensional case. It allows the curvature and
spinor terms to be compared in any complex dimension and, in the compact
case, preserves the variational origin of the equations. The first main
result can therefore be viewed as the higher-dimensional analogue of the
energy identity underlying the classical theory.

Therefore, let $\mathcal{C}_{SD}$ be the subset of the configuration space consisting of pairs $(\psi,A)$ such that $F_A^-=0$, and let $a_n$ be the dimension-dependent constant
\[a_n=\begin{cases}
4, & n=2,\\
2^{n-1}, & n\neq2.
\end{cases}\]
We define the following energy functional, which recovers the original \eqref{eq:sw-energy-intro} in the four-dimensional case.
\begin{equation}\label{eq:energy-intro2}
\mathcal{E}(\psi,A)=\int_{M}\left(
\bigl|\nabla_A\psi\bigr|^{2}
+\bigl|F_A\bigr|^{2}
+\frac{\scal}{4}|\psi|^{2}
+\frac{a_{n}\bigl|\gamma(\sigma(\psi)^{+})\bigr|^{2}}{32}
\right){\vol}_{M}.
\end{equation}

\medskip
\noindent\textbf{Theorem A.}
\emph{Let $G/K$ be a compact Hermitian symmetric
space of complex dimension $n$, and let $E_\nu$ be a one-dimensional
admissible twist. A pair $(\psi,A)$ satisfying
$F_{A}^-=0$ is an absolute minimum of the energy functional $\mathcal{E}_{{|}_{\mathcal{C}_{SD}}}$ if
\begin{equation} 
  D_{A}^+\psi=0,\qquad
  F_{A}^+=-\frac{a_n}{8}\sigma(\psi)^+,\qquad
  \Lambda_\omega F_{A}\ \text{is constant}.
  \label{eq:theorem-a-intro}
\end{equation}
The last condition in \eqref{eq:theorem-a-intro} is vacuous when $n=2$ and is automatic when $A\in\mathcal A^G$. We equip $G/K$ with a Riemannian metric corresponding to $\lambda {B_{\mathfrak g}}_{{|}_\fp}$, where $B_\fg$ is the Killing form and $\lambda$ is a negative real parameter. For an invariant connection, the
$\zeta$-perturbed monopole equation for $\zeta=it\omega$, where $t$ is a real parameter, is equivalent to
\begin{equation}
 \sum_{k=0}^{\lfloor n/2\rfloor}(n-4k)
       |\psi^{(2k)}|^2=c_n(t,\nu,\lambda),\qquad
 \langle\gamma(\eta_{ij})\psi,\psi\rangle
 =\langle\gamma(\theta_{ij})\psi,\psi\rangle=0,
 \label{eq:explicit-monopole-intro}
\end{equation}
where
\[ c_n(t,\nu,\lambda)=8n\left(t+\frac{\nu}{2\lambda}\right).\]
}

A proof of this result is a combination of \cref{theo:var} and \cref{theo:monopole}. Theorem~A makes the two parts of the problem visible. The Dirac equation is linear and representation-theoretically rich, while \eqref{eq:explicit-monopole-intro} imposes pointwise quadratic conditions. Let $\mathcal S_{\nu}^\infty$ denote the set of smooth solutions of the $it\omega$-perturbed equations of \cref{eq:theorem-a-intro} and $\mathcal M^\infty_{\nu}$ their moduli space; we denote by $\mathcal S^\infty_{\nu,{eq}}$, $\mathcal S^{L^2}_{\nu,{eq}}$ those for which $A\in\mathcal A^G$. % spinor lies in $\ker_{L^2}D_0^+$. 
In the compact case, the kernel of the full Dirac operator is a non-zero irreducible representation for every admissible $|\nu|\geq1$, and its dimension grows polynomially with $|\nu|$. In the non-compact case the corresponding discrete-series representation is infinite-dimensional. The prominent conclusion is that the quadratic monopole condition collapses almost all of this linear solution space. We denote by $\mathcal M^\infty_{\nu,{eq}}$, $\mathcal M^{L^2}_{\nu,{eq}}$ the corresponding moduli sets.% obtained by dividing the corresponding set of solutions by the equivariant gauge group. 

\medskip
\noindent\textbf{Theorem B.}
\emph{Let $G/K$ be irreducible of complex dimension $n$, let $E_\nu$ be a
one-dimensional admissible twist, and let $A\in\mathcal A^G$. If $G$ is
compact, then
\begin{equation}
 \mathcal M^\infty_{\nu,{eq}}(c_n)\cong
 \begin{cases}
   \{\mathrm{pt}\},&
       c_n>0,\ \nu=1,\\
   \{\mathrm{reducible\ pt}\},&
       c_n=0,\ \text{for every admissible }\nu,\\
   \{\mathrm{pt}\},&
       c_n<0,\ \nu=-1,\ n\ \text{even},\\
   \varnothing,&\text{otherwise}.
 \end{cases}
 \label{eq:theorem-b-intro}
\end{equation}
If $G$ is non-compact, then
$\mathcal M^{L^2}_{\nu,\mathrm{eq}}(c_n)$ is the reducible point when $c_n=0$ and
is empty otherwise.
}

\medskip
The compact statement, proved in \cref{theo:compactmoduli}, says more than finiteness: whenever the irreducible part of the equivariant moduli set is non-empty, it consists of a single gauge class. In the non-compact case, a non-zero solution of the monopole equation would have constant pointwise norm. Since $G/K$ has infinite
volume, such a spinor cannot be square-integrable; hence no non-zero element
of the $L^2$ Dirac kernel survives the nonlinear equation. This is the basic
rigidity phenomenon of the paper.

\medskip
The parity condition in Theorem~B is not an artefact of the calculation. It
comes from the representation carried by the Dirac kernel. Since $D_0$
commutes with the left action of $G$, its kernel is a $G$-representation. For
compact $G$, it is determined by the Borel--Weil--Bott theorem and its
Dirac-operator formulation \cite{Bo57,La00}. For non-compact $G$,
Parthasarathy proved that suitable $L^2$ kernels of twisted Dirac operators
are discrete-series representations \cite{Par72}. Atiyah and Schmid proved
the converse realisation: every discrete-series representation of $G$ is
obtained as the $L^2$ kernel of an appropriately twisted Dirac operator
\cite{AS77}. Thus the first Seiberg--Witten equation does not merely select harmonic
spinors; it places the monopole in an irreducible representation of the
transitive symmetry group.

\medskip
In the present setting this representation is described explicitly using
only the one-dimensional twists required by Seiberg--Witten theory. The same
Weyl group element controls the compact highest-weight realisation and the
discrete-series realisation on the non-compact dual. Computing its length
explains exactly when the representation occurs in the positive half-spinor
kernel.

\medskip
\noindent\textbf{Theorem C.}
\emph{Let $G/K$ be an irreducible Hermitian symmetric space of complex
dimension $n$. If $G$ is compact and $\nu\leq-1$, there is a unique
$w\in W_G$ for which
\[
  w(\nu\rho_{\mathfrak p}+\rho_{\mathfrak k})
       -\rho_{\mathfrak g}
\]
is dominant, and $\ell(w)=n$. Consequently,
\[
  \ker D^+(E_\nu)\neq0
  \quad\Longleftrightarrow\quad n\in2\mathbb Z.
\]
For $\nu\geq1$, the non-triviality holds unconditionally. If $G$ is non-compact, $\nu\geq1$, and
$\delta_\nu:=\nu\rho_{\mathfrak p}+\rho_{\mathfrak k}$, then
\[
  \ker_{L^2}D^+(E_{\nu\rho_{\mathfrak p}})
  \ \text{is a discrete-series representation}
  \quad\Longleftrightarrow\quad n\in2\mathbb Z.
\]
Whenever it occurs, its Harish--Chandra parameter is $\delta_\nu$;  for $\nu\leq -1$, this holds unconditionally. For
$|\nu|<1$, both non-triviality and the discrete-series realisation fail.
}

\medskip
Theorem~C is the representation-theoretic core of the paper and is of independent interest in relation to Seiberg--Witten equations; its proof is a combination of \cref{coro:neven} and \cref{coro:nevennc}. It identifies the Harish--Chandra parameter directly from the $\operatorname{Spin}^{c}$ twist and shows that the parity obstruction is the geometric shadow of the equality $\ell(w)=n$. It also clarifies how the nonlinear problem interacts with representation theory. The Dirac equation can supply an entire irreducible representation, finite-dimensional in the compact case and infinite-dimensional in the non-compact case, whereas the monopole equation selects only the vectors compatible with its rank-one quadratic constraint. In the compact case these reduce to the distinguished degree-zero or top-degree spinors appearing in Theorem~B; in the non-compact $L^2$ theory, every non-zero vector is excluded.

\medskip
There is also a complex-geometric interpretation of this collapse. In
complex dimension two, Witten showed that the K\"ahler Seiberg--Witten
equations decouple and are related to effective divisors \cite{W94}; the
parallel statement for vortices is Bradlow's theorem \cite{Br90}. In our
setting the equivariant moduli set, $\mathcal{M}_{\nu,eq}^\infty$, agrees with the $G$-fixed locus of the moduli space $\mathcal{M}_\nu^\infty$. Since an irreducible Hermitian symmetric space has no non-zero effective $G$-invariant divisors, Bradlow's trichotomy becomes the rigid trichotomy of Theorem~B.

\medskip
\noindent\textbf{Theorem D.}
\emph{Let $M=G/K$ be compact and irreducible, let $E_\nu$ be a one-dimensional admissible twist and let $\mu=\deg_\omega(\mathcal{E}_{\nu-1})$, $\nu\geq 1$. Then
\begin{equation}
 \bigl(\mathcal M_\nu^\infty(c_n)\bigr)^G\cong
 \begin{cases}
   \bigl(\operatorname{Div}^{\mu}_+(M)\bigr)^G,&c_n>0,\\
   \bigl(\operatorname{Hol}^{\mu}(M)\bigr)^G,&c_n=0,\\
   \varnothing,&c_n<0.
 \end{cases}
 \label{eq:theorem-d-intro}
\end{equation}}

\medskip
Here $\Div_+^\mu(M)$ denotes the set of effective divisors of degree $\mu$ and $\Hol^\mu(M)$ the set of holomorphic line bundles of degree $\mu$; the proof is found in \cref{equivreg}. The four theorems exhibit the structure of the argument. The invariant K\"ahler geometry produces a higher-dimensional self-duality operator and a variational monopole system; homogeneous Dirac theory converts the linear equation into a problem in representations of $G$; the explicit Weyl group calculation determines the relevant highest-weight and discrete-series representations; and the quadratic monopole equation forces the final rigidity. The four-dimensional K\"ahler case goes back to Witten \cite{W94}, while Seiberg--Witten equations on four-dimensional locally symmetric spaces were studied by Speh using representation theory \cite{Sp00}. The
representation-theoretic tools used here are those of Bott, Landweber,
Parthasarathy, and Atiyah--Schmid \cite{Bo57,La00,Par72,AS77}. Detailed
background and some of the case-by-case weight computations are developed
further in the author's doctoral thesis \cite{Tsi26}.
\medskip

The equivariant setting thus reveals the strong rigidity of the
Seiberg--Witten equations on Hermitian symmetric spaces: although the Dirac
equation may admit a rich space of solutions, the nonlinear monopole
constraint leaves very few possibilities. In the compact case, the
equivariant moduli set is either empty or a single gauge class, while in
the non-compact case square integrability excludes all non-zero monopoles.
Symmetry therefore makes both the equations and the scarcity of their
solutions explicit.

%\raggedbottom

\medskip
\begingroup
\makeatletter
\let\@tocwrite\@gobbletwo
\section*{Notational conventions}

\begin{convention*}\label{conventionHSS} $ $
\begin{enumerate}
    \item[(i)] From \cref{subsec:HSS} onwards, $G$ is a connected simple Lie group with finite center and $K$ is a subgroup with one-dimensional center, %maximal compact if $G$ is non-compact and proper connected of maximal rank if $G$ is compact, 
    so that $G/K$ is an irreducible Hermitian symmetric space. The reductive notation $\fg=\fh\,\oplus\,\fq$ is used only in \cref{wang} and \cref{subsec:dirac}; everywhere else $\fh=\fk$ and $\fq=\fp$.
    \item[(ii)] We fix the Cartan subalgebra $i\ft$, where $\ft$ is a maximal abelian subalgebra of $\fk$. Positive roots are chosen so that $\fp^{\pm}_{\C}=\bigoplus_{\alpha\in\Delta^{+}_{\fp}}\fg_{\pm\alpha}$ are the $\pm i$-eigenspaces of $J$, see \cite{Kna96} for reference.% In order to treat the compact and non-compact case simultaneously, we employ the unified notation %($\,\mathrm{{\txs{---}}_\fp}\,$) for $\rho_\fp$ and $\Delta_\fp$.
    \item[(iii)] Sections of vector-valued forms of homogeneous bundles are occasionally expressed in the form $f^{ij}\ox u_i\ox v_j$, as elements of $\{C^\infty(G)\ox U\ox V\}^K$.
    \item[(iv)] On endomorphisms we use the norm induced by the Hilbert--Schmidt inner product $\langle A,B\rangle=\tr(B^*A)$.
    %\item[(iv)] All statements about $E_{\nu}$ presuppose that $E_{\nu}$ is a one-dimensional admissible twist, and all statements about solutions presuppose $A\in\mathcal{A}^G$, unless said otherwise.
\end{enumerate}
\end{convention*}
\bigskip
\section*{Notation table}
\medskip
\small
\setlength{\tabcolsep}{7pt}
\renewcommand{\arraystretch}{1.18}
\begin{center}
\begin{tabular}{@{}ll@{\qquad}ll@{}}
$G/K$ & irreducible Hermitian symmetric space & $n$ & $\dim_{\C}G/K$, $\dim_{\R}\fp=2n$\\
$\fg=\fk\oplus\fp$ & Lie algebra decomposition & $Z_{0}$ & generator of $\fz_{\fk}$, $J=\ad(Z_0)$\\
$B_\fg$ & Killing form, $B_\fp={B_\fg}_{{|}_\fp}$ & $\lambda$ & metric parameter, $(\cdot,\cdot)_{\lambda}=\lambda B_{\fp}$  \\
$\rho_{\fg},\rho_{\fk},\rho_{\fp}$ & half-sums of positive roots & $N$ & dual Coxeter number, $\varpi_{\fp}=\tfrac{2}{N}\rho_{\fp}$\\
$\nu$ & admissible parameter, $\nu=1+2\kappa/N$ &$\mono\left(t,\nu,\lambda\right)$ & monopole constant $8n(t+\tfrac{\nu}{2\lambda})$ \\
$E_{\nu}$ & one-dimensional admissible twist $E_{\nu\rho_{\fp}}$ & $\mathcal{L}_{\nu}$ & $\mathcal{E}_{\nu}\otimes \mathcal{E}_{\nu}$, characteristic bundle\\
$S^{(\pm)}$ & (half-)spin modules, $\sprep^{(\pm)}$ the spin rep. & $m^{\pm}$ & $\dim(S^{\pm}\otimes E_{\nu})=2^{n-1}$\\
$\mathcal{K}$ & canonical bundle, $\mathcal{K}=E_{-2\rho_{\fp}}$ & $\Lambda_{\varphi}$ & contraction against $\varphi$\\
$\mathcal{A},\mathcal{A}^G$ & connections, $G$-invariant ones & $\Lambda$ & Wang map $\fg\to\mathfrak f$\\
$\omega$ & Kähler form, $|\,\omega\,|^{2}=n$ & $\Phi$ & canonical $(2n-4)$-form\\
$*_{\Phi}$ & twisted Hodge operator & $a_{n}$ & $4$, $n=2$; $2^{\,n-1}$, $n\neq2$\\
$\sigma(\psi)$ & $2$-form valued squaring map & $\mathcal{Q}(\psi)$ & $\psi\psi^{*}$, $\Quad_0(\psi)$ its traceless part\\
$\mathcal{E}$, $\mathcal{C}$ & energy functional, configuration space & $\mathcal{G}, \mathcal{S},\mathcal{M}$ & gauge group, solutions and moduli sets
\end{tabular}
\end{center}
\makeatother
\endgroup
\smallskip
\section{Preliminaries}

\subsection{Seiberg--Witten and vortex equations}\label{SWpara}

Let $(M,g)$ be a closed oriented Riemannian four-manifold equipped with a $\Spin^c$ structure $\mathfrak{s}$. Let also $\psi$ be a section of the positive complex spinor bundle $S_\mathfrak{s}^+$ of $\mathfrak{s}$, $F_{A}$ be the curvature of a connection $A$ on the determinant bundle of $\mathfrak{s}$ and $\nabla_A^+$ be the connection on $S_\mathfrak{s}^+$ induced by the Levi-Civita connection on $M$ and $A$. If $\gamma$ denotes the Clifford multiplication, we consider $D_{A}^+=\gamma\circ\nabla_A^+ $ to be the Dirac operator corresponding to $\nabla_A^+$ and $\sigma$ to be the following quadratic map
\begin{equation}\label{eq:sigma4}
\sigma:\Gamma(S_\mathfrak{s}^+)\rightarrow \Lambda^2T^*M, \qquad \sigma(\psi)=\frac{1}{4}\sum_{i<j}^{4}\bigl(\gamma(X^{*}_{i})\gamma(X^{*}_{j})\psi,\psi\bigr) X^{*}_{i}\w X^{*}_{j} .
\end{equation}
Since $M$ is four-dimensional, the Hodge-star operator $*$ is an involution on the space of two-forms. Let $\Lambda^{2,+}$ be its positive eigenspace, the \emph{self-dual} forms, and $F^{+}_{A}$ be the projection of $F_{A}$ on $\Lambda^{2,+}$. The extension of $\gamma$ on the space of two-forms induces a natural isomorphism between $\Lambda^{2,+}$ and the traceless skew-Hermitian endomorphisms of $S_\mathfrak{s}^+$, hence $\sigma(\psi)\in \Lambda^{2,+}$. If $\zeta\in\Lambda^{2,+}$, the $\zeta$-perturbed Seiberg--Witten equations are
\begin{equation}\label{sweq4}
D_{A}^+\psi=0,\qquad F_{A}^{+}+\zeta=-\frac{1}{2}\sigma(\psi).
\end{equation} 
As mentioned above, the appearance of the unperturbed equations is closely tied to their variational nature. Namely, their solutions attain the absolute minimum of the energy functional \cref{eq:sw-energy-intro}. The general version of this computation is carried out in \cref{sec:gen}, where the four-dimensional case is recovered as $n=2$. 
\medskip

We next introduce the vortex equations, which may be regarded as the two-dimensional precursor of the Seiberg--Witten equations, see \cite{GP94} for reference. On a closed Kähler manifold $(M,\omega)$ we consider 
\begin{equation*}
    {\nabla_A}^{(0,1)}\psi=0, \qquad {F_A}^{(0,2)}=0, \qquad i\Lambda_{\omega}F_A+\frac{1}{2}\psi\ox\psi^*=\frac{1}{2}t,
\end{equation*}
where $\psi$ is a section of a Hermitian line bundle $\mathcal{F}$ over $M$ and $A$ is a connection on $\mathcal{F}$. They also arise as the absolute minimum of an energy functional, called the Yang--Mills--Higgs functional. The moduli space of their solutions, $\mathcal{M}^{\mu}_{t}$, was described by Bradlow in the following theorem. Here ${\cong}$ denotes the one-to-one correspondence.
\begin{theorem}[\cite{Br90}, Theorems 4.6, 4.7]\label{Brad1} Let $(M,\omega)$ be a closed Kähler manifold and $\mu$ a non-negative integer. Let $\Div_+^\mu(M)$ be the set of effective divisors with degree $\deg_\omega$ equal to $\mu$ and $\Hol^\mu(M)$ the set of holomorphic line bundles of degree $\mu$.
We fix $t$ and set $b(\mu,t)=t-4\pi\mu/\vol(M)$. 
\begin{equation}
 \mathcal{M}_t^\mu\cong
 \begin{cases}
   \Div_+^{\mu}(M),&b(\mu,t)>0,\\
   \Hol^{\mu}(M),&b(\mu,t)=0,\\
   \varnothing,&b(\mu,t)<0.
 \end{cases}
 \end{equation}
\end{theorem}

\subsection{Homogeneous bundles and Wang's Theorem}\label{wang}

Let $G$ be a connected Lie group and $H$ a closed subgroup of $G$ with $G/H$ reductive, that is, there exists a subspace $\fq$ of $\fg$ such that $\fg=\fh\op\fq$ and ${\Ad(H)\fq\subseteq \fq}$. Let $(\cdot,\cdot)_\fq$ be an $\Ad(H)$-invariant inner product on $\fq$ such that $([X,Y]_\fq,Z)_\fq=-(Y,[X,Z]_\fq)_\fq$ for all $X,Y,Z\in\fq$. The scalar curvature of $G/H$ is constant and, for an orthonormal basis $\{X_{i}\}$ of $\fq$ with respect to $(\cdot,\cdot)_\fq$, 
\begin{equation}\label{eq:scal}
\scal=-\frac{1}{4}\sum_{i,j} ([X_i,X_j]_\fq,[X_i,X_j]_\fq)_\fq - \frac{1}{2} \sum_{i} B_\fq(X_i,X_i).
\end{equation}
If $B_\fg$ is the Killing form of $\fg$, we set ${B_\fq= B_{\fg}}_{|_{\fq}}$ and consider $(\cdot,\cdot)_{\fq,\lambda}\coloneqq\lambda B_{\fq}$, where $\lambda$ is a non-zero real parameter. In this case, we rewrite \cref{eq:scal} as follows; %one has $B_\fq(X_{i},X_{i})=1/\lambda$, so the second term equals $-\dim\fq/(2\lambda)$. 
we refer to %7.37b of 
\cite{Be87} for more details.
%for $ {B_\fq(X_i,X_i)}$
\begin{equation}\label{eq:scal2}
\scal(\lambda)=\lambda\Bigl(\sum_{i,j}B_{\fg}\bigl([X_{i},X_{j}]_{\fh},[X_{i},X_{j}]_{\fh}\bigr)
+\frac{1}{4}\sum_{i,j}B_{\fg}\bigl([X_{i},X_{j}]_{\fq},[X_{i},X_{j}]_{\fq}\bigr)\Bigr).
\end{equation}

\begin{remark}\label{rem:signlambda}
For $(\cdot,\cdot)_{\fq,\lambda}$ to be positive definite, $\lambda$ must be negative
when $G$ is compact and positive when $G$ is non-compact. This is the criterion by
which the two cases are distinguished throughout.
\end{remark}

A principal bundle $F\to P\overset{\pi}{\to} G/H$ is \emph{homogeneous} if the $G$-action on the base lifts compatibly with the $F$-action; %By Proposition 7.1.3 of \cite{ANT23}, 
this happens if and only if there is a Lie group homomorphism $\varphi\colon H\to F$, the \emph{isotropy homomorphism}, with $P\simeq G\times_{\varphi}F$. Wang's correspondence, \cite{Wan58}, then identifies the $G$-invariant connection one-forms on $P$, $\mathcal{A}^G$, with the $H$-equivariant linear maps
\[\Lambda\colon\fg\longrightarrow\mathfrak{f},\qquad {\Lambda}_{{|}_{\fh}}=d\varphi.\]
We refer to \cite{ANT23} for a detailed proof of both Wang's Theorem and the following formulas. We choose $p_{0}\in\pi^{-1}(eH)$ of $P$ and $\{W_i\}$ a basis of $\mathfrak f$. If $X_i^*$ and $W_j^*$ denote the duals of $(X_i,0)$, $(0,W_j)$ belonging to %by $\llbracket\cdot,\cdot\rrbracket$ the equivalence classes in 
$\fg\times_{\fh}\mathfrak{f}$, an invariant connection and its curvature at $p_0$ are given by
%A_{p_{0}}=\sum_{i=1}^{\dim\fq}1\otimes\llbracket X_{i},0\rrbracket^{*}\otimes\Lambda(X_{i})+\sum_{j=1}^{\dim\mathfrak{f}}1\otimes\llbracket0,W_{j}\rrbracket^{*}\otimes W_{j},
%F_{A_{p_{0}}}=\sum_{i<j}^{\dim\fq}1\otimes\llbracket X_{i},0\rrbracket^{*}\wedge\llbracket X_{j},0\rrbracket^{*}\otimes\Bigl([\Lambda(X_{i}),\Lambda(X_{j})]_{\mathfrak{f}}-d\varphi([X_{i},X_{j}]_{\fh})-\Lambda([X_{i},X_{j}]_{\fq})\Bigr).
\begin{equation}\label{eq:Aformula}
A_{p_{0}}=\sum_{i=1}^{\dim\fq}1\ox X_i^{*}\ox\Lambda(X_{i})
+\sum_{j=1}^{\dim\mathfrak{f}}1\ox W_j^{*}\ox W_{j},
\end{equation}
\begin{equation}\label{eq:Fformula}
F_{A_{p_{0}}}=\sum_{i<j}^{\dim\fq}1\ox X_i^{*}\wedge X_j^{*}
\otimes\Bigl([\Lambda(X_{i}),\Lambda(X_{j})]_{\mathfrak{f}}-d\varphi([X_{i},X_{j}]_{\fh})
-\Lambda([X_{i},X_{j}]_{\fq})\Bigr).
\end{equation}
\begin{remark}\label{FAabelian}
At any other point $p$, ${F_{A}}_p$ is given by the image of ${F_{A}}_{p_0}$ under the adjoint action of the Lie group $F$. If $F$ is the abelian $U(1)$, the curvature of an invariant connection is therefore constant.
\end{remark}

\subsection{Hermitian symmetric spaces}\label{subsec:HSS}
We now impose Convention \ref{conventionHSS}, $(G/K,\lambda B_\fp)$ is assumed to be a Hermitian symmetric space, where $\fg=\fk\op \fp$ is the corresponding Lie algebra decomposition. We denote by $\Delta_\fg$ the roots of $\fg$ and by $\Delta_\fk$ its subset containing those which vanish on the center $\fz_\fk$ of $\fk$. We set $\Delta_\fp= \Delta_{\fg}\setminus\Delta_{\fk}$ and henceforth refer to $\Delta_\fp$ as the set of $\fp$-roots. Up to scaling of the invariant metric, every irreducible Hermitian symmetric space belongs to one of the following families; see \cite{Hel78} for reference.
\begin{table}[h]
\begingroup
\small
\begin{center}
\renewcommand{\arraystretch}{1.2}
\begin{NiceTabular}{|c|c|c|c|}[hlines]
\textbf{Type} & \textbf{Non-Compact} & \textbf{Compact} & \textbf{Real Dimension} \\  
A\txs{III} (${p\geq q}$) & ${\SU(p,q)/\Si(\U(p)\x\U(q))}$ & ${\SU(p+q)/\Si(\U(p)\x\U(q))}$
& ${2pq}$ \\ 
BD\txs{I} (${p>2}$) & ${\SO_e(2,p)/(\SO(2)\x\SO(p))}$ & ${\SO(2+p)/(\SO(2)\x\SO(p))}$ & ${2p}$ \\ 
D\txs{III} & ${\SO^*(2p)/\U(p)}$ & ${\SO(2p)/\U(p)}$ & ${p(p-1)}$ \\ 
C\txs{I} & ${\Sp(p,\R)/\U(p)}$ & ${\Sp(p)/\U(p) }$ & ${p(p+1)}$ \\ 
E\txs{III} & $ {\text{E}_{6(-14)}/(\Spin(10)\x_{\Z_4}\U(1))} $ & $ {\text{E}_{6(-78)}/(\Spin(10)\x_{\Z_4}\U(1)) }$ & 32 \\ 
E\txs{VII} & $ {\text{E}_{7(-25)}/(\text{E}_{6}\x_{\Z_3}\U(1)})  $& $ {\text{E}_{7(-133)}/(\text{E}_{6}\x_{\Z_3}\U(1)) }$ & 54 \\
\end{NiceTabular}
\end{center}
\endgroup
\end{table}

These spaces satisfy $[\fp,\fp]_{\fp}=0$ and are characterised among symmetric spaces by the property $(Z(K))_0\simeq U(1)$; here $(Z(K))_0$ denotes the identity component of the center of $K$. Equivalently, there exists a unique, up to sign, element $Z_{0}\in\fz_{\fk}$ for which $J=\ad(Z_0)$ is the $K$-equivariant complex structure. A second key property of such spaces is that the group $K$ acts irreducibly on each of $\fp^{\pm}_{\C}$; we refer to \cite{Kna96} for both. 
\medskip

We consider the ordering specified in Convention \ref{conventionHSS} and denote by $\Delta_\fg^+$ the set of positive roots and by $\rho_\fg$ their half-sum. For this ordering, 
\begin{equation}\label{aZ0}
    \alpha(Z_0)=\begin{cases}
    i & \text{ for all }\alpha\in\Delta_{\fp}^+\\
    -i & \text{ for all }\alpha\in\Delta_{\fp}^-\\
    0 & \text{ for all }\alpha\in\Delta_{\fk}.
\end{cases}
\end{equation}
If $\varpi_{\fp}$ is the fundamental weight of the unique simple $\fp$-root and $N$ the dual Coxeter number, that is, the number such that $\varpi_{\fp}=\tfrac{2}{N}\rho_{\fp}$, where $\rho_{\fp}$ is the half-sum of positive $\fp$-roots, we obtain the following data.
\begin{table}[h]
\begingroup
\small
\begin{center}
\renewcommand{\arraystretch}{1.5}
    \fontsize{8.6}{10}\selectfont
    \begin{NiceTabular}{|c|c|c|c|c|c|c|}[hlines]
\textbf{Type} & \textbf{A\txs{III}} & \textbf{BD\txs{I}} & \textbf{D\txs{III}} & \textbf{C\txs{I}} & \textbf{E\txs{III}} & \textbf{E\txs{VII}} \\
$N$ & $p+q$ & $p$ & $2(p-1)$ & $p+1$ & $12$ & $18$ \\
$\rho_\fp$
& $\tfrac{q}{2}(\gve_1+\ldots+\gve_p)
   -\tfrac{p}{2}(\gve_{p+1}+\ldots+\gve_{p+q})$ & $\tfrac{p}{2}\gve_1$& $\tfrac{p-1}{2}\sum_{i=1}^{p}\gve_i$ & $\tfrac{p+1}{2}\sum_{i=1}^{p}\gve_i$ & $6\varpi_6$ & $9\varpi_7$ \\
$\varpi_\fp$ & $\tfrac{q}{p+q}(\gve_1+\ldots+\gve_p)
   -\tfrac{p}{p+q}(\gve_{p+1}+\ldots+\gve_{p+q})$ & $\gve_1$ & $\tfrac{1}{2}(\gve_1+\ldots+\gve_p)$ & $\gve_1+\ldots+\gve_p$ & $\varpi_6$ & $\varpi_7$ \\
$iZ_0^*$ & $\tfrac{1}{p}(\gve_1+\ldots+\gve_p)
   -\tfrac{1}{q}(\gve_{p+1}+\ldots+\gve_{p+q})$ & $\gve_1$ & $\tfrac{2}{p}\sum_{i=1}^{p}\gve_i$ & $\tfrac{2}{p}\sum_{i=1}^{p}\gve_i$ & $\tfrac{3}{4}\varpi_6$ & $\tfrac{2}{3}\varpi_7$ \\
$n$ & $pq$ & $p$ & $\tfrac{p(p-1)}{2}$ & $\tfrac{p(p+1)}{2}$ & $16$ & $27$
\end{NiceTabular}
\end{center}
%\caption{Hermitian data}\label{Hermdata}
\endgroup
\end{table}
\vspace{-1mm}
\begin{lemma}\label{Z_0*rho}
If $Z_{0}^*\in\fz_{\fk}^*$ denotes the dual of $Z_{0}$, then $iZ^{*}_{0}=\tfrac2n\rho_{\fp}$. In particular, $\rho_{\fp}\perp\Delta_\fk$.
\end{lemma}
\begin{proof}
    Follows directly from \cref{aZ0} and the table above.
\end{proof}
\begin{lemma}\label{omegavarpirel}
We consider the K\"ahler form $\omega_\lambda(X,Y)=\lambda B_{\fp}(JX,Y)$. For all $X,Y\in\fp$, 
\[\rho_{\fp}\bigl([X,Y]\bigr)=-\frac{i}{4\lambda}\,\omega_\lambda(X,Y)
\qquad\iff \quad \omega_\lambda(X,Y)=2\lambda N i\,\varpi_{\fp}(\ad(X)Y).\]
\end{lemma}
\begin{proof}
By \cref{Z_0*rho} and $B_{\fg}(Z_0,Z_0)=-2n$, we have $\rho_{\fp}([X,Y])=-\tfrac{i}{4}B_{\fg}([X,Y],Z_{0})$. The Killing form is invariant under the adjoint action, thus $B_{\fg}([X,Y],Z_{0})=B_{\fg}(JX,Y)=B_{\fp}(JX,Y)$. %The Lemma follows from the definition of $\omega_\lambda$.
\end{proof}

\begin{lemma}\label{LiftTable}
    Let $\xi\in \Q$, $o(\cdot)$ denote the order of an element, and let $\ell =\big |(Z(K))_0\, \cap \, K_{ss}\big|$, where $K_{ss}=[K,K]$. Then $\xi\rho_\fp$ %lifts to $(Z(K))_0^\C$ if and only if $(\xi\cdot n \cdot o(e^{2\pi Z_0}))/2$ is an integer and 
    lifts to $K$ if and only if $(\xi n  o(e^{2\pi Z_0}))/(2\ell)$ is an integer.
   % \begin{align*}     {\xi\rho_\fp \text{ lifts to ${(Z(K))_0}^\C$ }}&\iff {\tfrac{\xi\cdot n \cdot o(e^{2\pi Z_0})}{2}\in \Z}\\   {\xi\rho_\fp\quad \text{ lifts to ${K^\C}$ }}&\iff {\tfrac{\xi\cdot n \cdot o(e^{2\pi Z_0})}{2\cdot \ell}\in \Z.} \end{align*} \vspace{2mm}
Explicitly, the character $\chi_{\xi\rho_{\fp}}$ lifts to $K$ precisely under the following conditions. The table reads $\xi N/2\in\Z$ in all types except D\txs{III}, where it reads $\xi N/4\in\Z$.
\begin{table}[h]
\begingroup
\small
\begin{center}
\renewcommand{\arraystretch}{1.4}
\begin{NiceTabular}{c|c|c|c|c|c|c}[hvlines]
\textbf{Type} & A\txs{III} & BD\txs{I} & D\txs{III} & C\txs{I}  & E\txs{III} & E\txs{VII} \\
\textbf{Condition} & ${\tfrac{\xi  (p+q)}{2}\in \Z}$ & ${\tfrac{\xi p}{2}\in \Z}$ & ${\tfrac{\xi (p-1)}{2}\in \Z}$ & ${\tfrac{\xi (p+1)}{2}\in \Z}$ & $6\xi\in \Z$ & $9\xi\in \Z$ 
\end{NiceTabular}
\end{center}
\endgroup
\vspace{-1mm}
\end{table}
\end{lemma}
\begin{proof}
    According to \cite{Kna96} Proposition 4.58, a linear form $\zeta$ lifts to the identity component $(Z(K))_0$ if and only if $\zeta(Z)\in 2\pi i\Z$ for all $Z\in \exp^{-1}(1_{(Z(K))_0})$. By Theorem 4.29 of the same reference, $K=(Z_{K})_0 K_{ss}$. Therefore every such element $Z$ is of the form $\kappa 2\pi Z_0 \, o(e^{2\pi Z_0})/\ell$, where $\kappa$ is an integer. From \cref{Z_0*rho}, $\xi\rho_\fp$ lifts if and only if
    \begin{align*}
        \xi\rho_\fp\left(\kappa 2\pi Z_0 \, o(e^{2\pi Z_0})/\ell\right)\in 2\pi i\Z \iff  (\xi n  o(e^{2\pi Z_0}) )/(2\ell)\in \Z.
    \end{align*}
 To obtain the stated criterion, we substitute the following data in each case.
    \begin{center}
    \small
\renewcommand{\arraystretch}{1.2}
\begin{NiceTabular}{|c|c|c|c|c|c|c|}[hlines]
\textbf{Type} & A\txs{III} & BD\txs{I} & D\txs{III} & C\txs{I}  & E\txs{III} & E\txs{VII}\\
${o(e^{2\pi Z_0})}$ & ${\tfrac{p+q}{gcd(p,q)}} $ & $1$ & $2$ & $2$  & $3$ & $2$\\
${(Z(K))_0 \, \cap K_{ss}}$  
& ${\Z_{lcm(p,q)}}$ & ${\{I_{2+ p}\}}$ & ${\Z_p }$ & ${\Z_p}$ & ${\Z_4}$ & ${\Z_3}$\\
\end{NiceTabular}
\end{center}
    \end{proof}

\subsection{Spin representations and admissible twists}\label{sec:spin} Twisted spin bundles are treated from the spin-geometric side in \cite{LM89,Ni00} and from the representation-theoretic side in \cite{Par72, HP06}. The following conventions are chosen to bridge the two. We consider the dual product of $(\cdot,\cdot)_{\fp,\lambda}$, which we denote by $(\cdot,\cdot)_\lambda$. We fix an orthonormal basis $\{X^{*}_{1},Y^{*}_{1},\dots,X^{*}_{n},Y^{*}_{n}\}$ with respect to $(\cdot,\cdot)_\lambda$,
such that ${\fp_\C^{*}}^{\pm}=\langle X^{*}_{\pm\alpha_{j}}\rangle_{j=1}^{n}$ with
$X^{*}_{\pm\alpha_{j}}=\tfrac{1}{\sqrt2}(X^{*}_{j}\pm iY^{*}_{j})$. Let 
\[S=\Lambda\fp^{*}_{\C}{}^{-} \qquad \text{ equipped with }\qquad \gamma\colon\fp_\C^{*}\to\End(S), \qquad \gamma=\sqrt2\,\varepsilon-\sqrt2\,\iota\] where
$\varepsilon$ and $\iota$ are exterior and interior multiplication. For the $\C$-bilinear extension $(\cdot,\cdot)_\lambda^\C$, we denote by $R_{u,v}$ the elements $(v,\cdot)_\lambda^\C u-(u,\cdot)_\lambda^\C v$ and by $(Cl(\fp_\C^{*}),[\cdot,\cdot])$ the Clifford algebra of $\fp_\C^*$. The composition
\[\fk\xrightarrow{\ \ad^{*}\ }\fso(\fp_\C^{*})\xrightarrow{\ \varphi\ }
Cl(\fp_\C^{*})\xrightarrow{\ \gamma\ }\End(S),\qquad
\varphi(R_{u,v})=-\frac{1}{4}[u,v],\]
is the \emph{spin representation} $\sprep$ of $\fk$. Its restrictions to
$S^{+/-}=\Lambda^{\mathrm{ev/odd}}\fp^{*}_{\C}{}^{-}$ are the half-spin representations, written as $\sprep^{(\pm)}$ when the distinction is immaterial. We refer to \cite{Tsi26} for the proofs of the following statements.

\begin{proposition}\label{spinrepformula}
    The spin representation is given by the following formulas.
    \begin{align*}
	    sp(Y)=\sum_{\alpha_j>0}\alpha_j(Y)\left(\frac{1}{2}\gamma(X_{\alpha_j}^*)\gamma(X_{-\alpha_j}^*)+\frac{1}{2}\id\right),\ Y\in\ft, \quad sp(Y)=\frac{1}{2}\sum_{i<j}([Y,X_i],X_j)_{\lambda}\gamma(X_i^*)\gamma(X_j^*), \ Y\in \fk.
	\end{align*}
\end{proposition}

\begin{proposition}\label{uweights}
        The vector ${u_I=X_{-\alpha_{i_1}}^*\w \cdots \w X_{-\alpha_{i_s}}^*, \, I=\{ 1\leq i_1<\cdots<i_s \leq n\}}$, where ${u_{\varnothing}\coloneqq 1}$, is a $\ft_\fk$-vector of weight 
        \[ {-\rho_\fg+\rho_\fk+\sum_{j\in I}\alpha_j=\frac{1}{2}\sum_{j\in I}\alpha_j-\frac{1}{2}\sum_{j\notin I}\alpha_j}.\]
    \end{proposition}

\begin{corollary}
    \label{prop:weights}
The vectors $u_{\varnothing}$ and $u_{\mathrm{top}}$ have weights $-\rho_{\fp}$ and $\rho_{\fp}$.
\end{corollary}
\begin{proposition}\label{sptraceless}
For $n\ne1$ the representations $\sprep^{(\pm)}$ are traceless. For $n=1$ the representation $\sprep$ is traceless while $\tr(\sprep^{\pm})=\mp\rho_{\fp}$.
\end{proposition}

\begin{definition}\label{admissibledef}
A finite-dimensional irreducible $\fk$-representation $(E_{\nu},\pi_{\nu})$ is an \emph{admissible twist} if $(S\ox E_{\nu},\sprep\ox\pi_{\nu})$
lifts to a representation $\beta^{\nu}$ of $K$, that is, $d\beta^{\nu}=\sprep\ox\pi_{\nu}$.
\end{definition}
The pair $(S,\sprep)$ need not integrate to $K$, and neither need the half-spin representations; this is remedied by twisting. Since $K$ is connected and $\sprep^\pm\ox \pi_\nu$ are subrepresentations of $\sprep\ox \pi_\nu$, they also integrate to $K$. We write $\beta^{\nu,\,\pm}$ for the restriction and $m^\pm$ for $\dim(S^\pm\ox E_\nu)$. Admissible twists play the role of higher homogeneous spin structures in the equivariant setting. The one-dimensional ones correspond to $\Spin^c$-structures, and they always exist.
\begin{proposition}\label{Liftinglemma}
If $S^{(\pm)}$ does not lift to $K$, there exists a one-dimensional admissible twist $E_{\nu}:=E_{\nu\rho_{\fp}}$. Every such twist has $\nu=1+2\kappa/N$, with $\kappa\in2\Z$ if $G/K$ is of type D\txs{III} and $\kappa\in\Z$ otherwise.
\end{proposition}
\begin{proof}
      From \cite{Kna86}, $S\simeq\Lambda\fp^{+}\otimes(\Lambda^{\mathrm{top}}\fp^{-})^{1/2}$ with $\fk$ acting by the adjoint action on each factor. The first factor lifts to $K$ because $\fp^{+}$ does. If $(\Lambda^{\mathrm{top}}\fp^{-})^{1/2}=\mathbb{C}_{-\rho_{\fp}}$ does not lift, consider a one-dimensional $\fk$-representation $E_{\chi}=V_{0}\ox\C_{\chi}$ with $V_{0}$ trivial on $\fk_{ss}$. Since $\fz_{\fk}$ is one-dimensional, $\chi=\nu\rho_{\fp}$ for some $\nu\in\R$. By \cref{LiftTable} there is always $\nu\in\Q$ for which $\C_{-\rho_{\fp}}\ox E_{\nu\rho_{\fp}}\simeq\C_{(-1+\nu)\rho_{\fp}}$ lifts to $K$, namely $\nu=1+2\kappa/N$ with $\kappa$ as stated. The same argument applies to $S^{(\pm)}$.
\end{proof}

\begin{example}[$\fg=\fsu(3)$]\label{ex:su3}
Here $\fk=\mathfrak{s}(\fu(2)\op\fu(1))$ and $Z_{0}=\diag(i/3,i/3,-2i/3)$. The lexicographic ordering with respect to $\{H_{1}=-iZ_{0},H_{2}=i\,\diag(-i,i,0)\}$ gives $\gve_1>\gve_{2}>\gve_{3}$, so
$\Delta_\fp^{+}=\{\gve_{1}-\gve_{3},\gve_{2}-\gve_{3}\}$,
$\Delta_\fk^{+}=\{\gve_{1}-\gve_{2}\}$ and
$\rho_{\fp}=\tfrac{1}{2}\gve_{1}+\tfrac{1}{2}\gve_{2}-\gve_{3}$. Then, writing $S_{j}=S_{ss}^{j}\otimes S_{\fz_\fk}^{j}$ for an irreducible summand
\[
S=V_{0}\otimes\C_{-\rho_\fp}\oplus V_{1/2(\gve_1-\gve_2)}\otimes\C_{0}\oplus V_{0}\otimes\C_{\rho_\fp}
\simeq\bigl(V_{0}\otimes\C_{0}\oplus V_{1/2(\gve_1-\gve_2)}\otimes\C_{\rho_\fp}\oplus V_{0}\otimes\C_{2\rho_\fp}\bigr)\otimes\C_{-\rho_\fp},
\]
and $E_{\nu}$ is admissible if and only if $\nu=1+2\kappa/3$, $\kappa\in\Z$. Analogous decompositions for $\fso(5)$ and $\fso(6)$, as well as their non-compact duals, %obtained by the same procedure from Lemma 9.3.2 of \cite{Wal88}, 
are recorded in \cref{exatable}.
\end{example}
\vspace{-3mm}
\begin{table}[h]
\begingroup
\fontsize{9.5}{10}\selectfont
\begin{center}
\renewcommand{\arraystretch}{1.5}
\begin{NiceTabular}{c|c|c|c|c|c}
\hline
\multicolumn{1}{|c}{$n$} & $\fg$ & $\fk$ & $\rho_\fp$ & decomposition of $S^{(\pm)}$ & \multicolumn{1}{c|}{admissible $\nu$}
 %\\ ${G}$ & ${\SU(p+q), \ p\geq q}$ & ${\SO_e(2+p)}$ & ${\SO(2p)}$ & ${\Sp(p)}$ 
\\
\hline
\multicolumn{1}{|c}{2} & $\fsu(3)$ & $\mathfrak{s}(\fu(2)\op\fu(1))$& 
 $\tfrac{1}{2}\gve_{1}+\tfrac{1}{2}\gve_{2}-\gve_{3}$ &  $S=V_{0}\C_{-\rho_{\fp}}\op V_{1}\C_{0}\op V_{0}\C_{\rho_{\fp}}$ & \multicolumn{1}{c|}{$1+\tfrac{2\kappa}{3}$}  
\\
\hline
\multicolumn{1}{|c}{3} & $\fso(5)$ & $\fso(2)\op\fso(3)$ & $\tfrac{3}{2}\gve_{1}$ & $S=V_{0}\C_{\rho_{\fp}}\op V_{2}\C_{-1/3\rho_{\fp}}\op V_{0}\C_{-\rho_{\fp}}\op V_{2}\C_{{1/3}\rho_{\fp}}$  &\multicolumn{1}{c|}{$1+\tfrac{2\kappa}{3}$ } \\
\hline
\multicolumn{1}{|c}{4} & $\fso(6)$ & $\fso(2)\op\fso(4)$ & $2\gve_{1}$ &  $S^{+}=V_{(0,0)}\C_{\rho_{\fp}}\oplus V_{(2,0)}\C_{0}\oplus V_{(0,2)}\C_{0}\oplus V_{(0,0)}\C_{-\rho_{\fp}}$& \multicolumn{1}{c|}{ $1+\tfrac{\kappa}{2}$} \\
\hline
%\multicolumn{4}{c}{} & $S^{-}=V_{(1/2,1/2)}\C_{\chi_{1/2}}\oplus V_{(1/2,1/2)}\C_{\chi_{-1/2}}$ & \\ \cline{5}
\end{NiceTabular}
\end{center}
\vspace{2mm}
\caption{Decomposition of the (half-)spin modules in the lowest-dimensional examples. %The first row is \cref{ex:su3}. 
Highest weights for $\fk_{ss}$ actions are given with respect to its fundamental weights.}
\label{exatable}
\endgroup
\end{table}
\vspace{-7mm}
\begin{remark}\label{rem:liftfails}
Liftability is not detected factor-wise: if $S_{ss}$ lifts to $K_{ss}$ and $S_{\fz_{\fk}}$ lifts to $K$, then $S_{ss}\ox S_{\fz_{\fk}}$ lifts to $K$, but the converse fails. For example, if $\fg=\fsu(3)$ the summand $V_{1}\ox\C_{\rho_\fp}$ lifts although $\C_{\rho_\fp}$ does not. Likewise, $S^{(\pm)}\otimes E_{\nu}$ may lift to $K$, while neither factor does. For example, if $\fg=\fso(5)$ this happens for every admissible $\nu=1+2\kappa/3$.
\end{remark}
We often refer to the possibly not well-defined bundles $ {\mathcal{S^{(\pm)}}=G\x_{K} S^{(\pm)}}$, $ {\mathcal{E}_\nu=G\x_{K} E_{\nu}}$ as \emph{virtual bundles}. The virtual line bundle $ {\mathcal{E}_{\nu}}$ is most commonly described as the ``square root" of $ {\mathcal{E}_{\nu}\ox\mathcal{E}_{\nu}}$ and its curvature $ {F_{A^{E_\nu}}}$ as half of the curvature $ {F_{A^{E_\nu\ox E_\nu}}}$ of the induced product connection. That is because the \emph{characteristic} line bundle $ {\mathcal{L}_\nu\coloneqq \mathcal{E}_{\nu}\ox\mathcal{E}_\nu}$ is a well-defined homogeneous bundle because of the proposition below.

\begin{proposition}\label{2nu}
If $(E_{\nu},\pi_{\nu})$ is a one-dimensional admissible twist, then
$(E_{\nu}\otimes E_{\nu},\pi_{\nu}\otimes\pi_{\nu})$ lifts to $K$.
\end{proposition}

\begin{proof}
If $\C_{(-1+\nu)\rho_{\fp}}$ lifts, then $\nu-1$ satisfies the condition of \cref{LiftTable}, which in all types has the form $\xi \ell /2\in\Z$ for a fixed
integer $\ell$; $\ell$ is $N\in\Z $ for all types except D\txs{III}, where $\ell=N/2=p-1\in\Z$. Since $(\nu-1)\ell/2=k\in\Z$ then $\nu=(2k+\ell)/\ell$ and
$2\nu\ell/2=2k+\ell\in\Z$, so $2\nu$ satisfies the same condition.
\end{proof}

This last proposition clarifies the relation between the one-dimensional admissible twist, the characteristic bundle, and the determinant bundle of the $\Spin^c$ structure mentioned in \cref{SWpara} and the introduction. This can be summarised in the following equivalence. For $n=1$, this holds for the full spin module.
\begin{equation}\label{linebndleequiv}
\mathcal{E}_{\nu}\simeq\mathcal{L}_{\nu}^{1/2}\simeq det\bigl(\mathcal{S}^{(\pm)}\ox \mathcal{E}_{\nu})^{1/m^{(\pm)}}
\iff
\mathcal{E}_{\nu}^{\ox m^{(\pm)}}\simeq\mathcal{L}_{\nu}^{\ox m^{(\pm)}/2}
\simeq\det\bigl(\mathcal{S}^{(\pm)}\ox \mathcal{E}_{\nu}\bigr).
\end{equation}
\begin{proposition}\label{detnu}
Let $(E_{\nu},\pi_{\nu})$ be a one-dimensional admissible twist. % and $m^{(\pm)}=\dim(S^{(\pm)}\otimes E_{\nu})$. 
For $n\neq 1$, $\pi_{\nu}=\tfrac{1}{m^{(\pm)}}\tr\bigl(d\beta^{\nu,(\pm)}\bigr)$. For $n=1$,
$\pi_{\nu}=\tfrac1m\tr(d\beta^{\nu})$ and $\pi_{\nu\mp1}=\tfrac{1}{m^{\pm}}\tr(d\beta^{\nu,\pm})$.
\end{proposition}

\begin{proof}
A weight basis of $S^{(\pm)}\ox E_{\nu}$ consists of vectors $u\otimes v$ for a fixed $0\ne v\in E_{\nu}$, on which $d\beta^{\nu,(\pm)}(Y)$ acts by
$\sprep^{(\pm)}(Y)u\ox v+u\ox\pi_{\nu}(Y)v$. By \cref{sptraceless} the first summand is traceless for $n\ne1$, so the trace is $m^{(\pm)}\pi_{\nu}(Y)$. For $n=1$ the correction is $\mp\rho_{\fp}$.
\end{proof}

\begin{proposition}
\label{1-1cor}
Let $\mathfrak{s}_0$ denote the homogeneous $\Spin^c$ structure with $\mathcal{K}^*=\mathcal{E}_{2\rho_\fp}$ as its characteristic bundle. Then, with $\kappa\in2\Z$ for type D\txs{III} and $\kappa\in\Z$ otherwise, the assignment
\begin{align*}
    \{ \text{one-dimensional admissible twists }{E_{\nu}}\} & \quad {\longrightarrow} \quad \{ \text{homogeneous $\Spin^{c}$ structures}\}/\sim\\
    {E_\nu, \, \nu=1+2\kappa/N}\hspace{20mm} & \quad \longmapsto \hspace{20mm} {\mathfrak{s}_0\ox E_{2\kappa/N}}
\end{align*}
is a bijection between one-dimensional admissible twists and homogeneous $\Spin^c$ structures up to equivariant isomorphism.
\end{proposition}
\begin{proof}
Both sides are torsors over $\Hom(K,U(1))$. For the admissible twists, if $\nu=1+2\kappa/N$ and $\nu'=1+2\kappa'/N$ then $E_{\nu'}=E_{\nu}\otimes E_{\mu}$ with $\mu=2(\kappa'-\kappa)/N$, and by \cref{LiftTable} $E_{\mu}$ lifts, so $\mu\rho_{\fp}$ integrates to an element of $\Hom(K,U(1))$. For the homogeneous $\Spin^c$ structures, if $\mathfrak{s},\mathfrak{s}'$ are homogeneous with $\mathfrak{s}'=\mathfrak{s}\otimes E$, then $E$ must be homogeneous, hence
$E=G\times_{\chi}\C$ for some $\chi\in\Hom(K,U(1))$.
\end{proof}
\begin{remark}\label{rem:basepoint}
The base point $\mathfrak{s}_{0}$ corresponds to $S\otimes E_{1}$, with which it
shares the characteristic bundle $\mathcal{K}^{*}=E_{2}$. Moreover, when $n=2$ one has $m^{\pm}=2$ and \cref{linebndleequiv} gives $\mathcal{L}_{\nu}\simeq\det(\mathcal{S}^{\pm}\ox \mathcal{E}_{\nu})$; so $\mathfrak{s}_{0}$ is the $\Spin^c$ structure with determinant bundle $\mathcal{K}^{*}$, as in the classical formulation. %Finally, every $\Spin^c$ structure on $G/K$ is homogeneous, since $\Pic(G/K)\simeq H^{2}(G/K,\Z)$ is $\Z$ in the compact case and $0$ in the non-compact case, so that $\Hom(K,U(1))\to H^{2}(G/K,\Z)$ is surjective.
\end{remark}

\subsection{Invariant connection and its curvature}\label{subsec:curvature}

Let $(E_{\mu},\pi_{\mu})$ be a finite-dimensional admissible twist, so that there is $\beta^{\mu,(\pm)}\colon K\to GL(S^{(\pm)}\otimes E_{\mu})$ with
$d\beta^{\mu,(\pm)}=\sprep^{(\pm)}\otimes\pi_{\mu}$. Invariant connection $1$-forms on
$P=G\times_{\beta^{\mu,(\pm)}}GL(S^{(\pm)}\otimes E_{\mu})$ correspond to elements of
$\Hom_{K}\bigl(\fg,\End(S^{(\pm)})\otimes\End(E_{\mu})\bigr)$ restricting to
$\sprep^{(\pm)}\ox{\pi}_{\mu}$ on $\fk$. A natural such element is $\Lambda^{s}_{S^{(\pm)}}\ox1+1\ox \Lambda_{E_{\mu}}$, where $\Lambda_{E_{\mu}}$ is any element of $\Hom_{\fk}(\fg,\End(E_{\mu}))$ with ${\Lambda_{E_{\mu}}}_{{|}_{\fk}}=\pi_{\mu}$ and
\begin{equation}\label{eq:equivariance}
\Lambda_{E_{\mu}}([Y,X])=\pi_{\mu}(Y)\circ\Lambda_{E_{\mu}}(X)-\Lambda_{E_{\mu}}(X)\circ\pi_{\mu}(Y), \qquad X\in\fp,\ Y\in\fk,
\end{equation}
and where, for $X\in\fp$ and a real parameter $s$,
\begin{equation}\label{eq:LambdaS}
\Lambda^{s}_{S^{(\pm)}}(X)=\frac12\sum_{i<j}^{2n}\bigl(s[X,X_{i}]_{\fp},X_{j}\bigr)_{\lambda}\gamma(X^{*}_{i})\gamma(X^{*}_{j})\ \in\ \End(S^{(\pm)}), \qquad {\Lambda^{s}_{S^{(\pm)}}}_{{|}_{\fk}}=\sprep^{(\pm)}.
\end{equation}
Specialise now to a one-dimensional admissible twist $(E_{\nu},\chi_{\nu})$. Since the induced morphism after composing with the adjoint action is trivial, \cref{eq:equivariance} reduces to $\Lambda_{E_{\nu}}([Y,X])$ $=0$. As $K$ acts irreducibly on $\fp$, the latter implies that
\begin{equation}\label{eq:Lambdavanish}
{\Lambda_{E_{\nu}}}_{{|}_{\fp}}\equiv0,
\qquad\text{and}\qquad
{\Lambda^{s}_{S^{(\pm)}}}_{{|}_{\fp}}\equiv0\ \ \text{because}\ [\fp,\fp]_{\fp}=0 .
\end{equation}
Substituting \cref{eq:Lambdavanish} into \cref{eq:Fformula}, all bracket terms lying in $\fp$ drop out and
\begin{equation}\label{eq:Ftwisted}
F_{A_{S^{(\pm)}\otimes E_{\nu}}}(X_{i},X_{j})
=\bigl(-\sprep^{(\pm)}([X_{i},X_{j}]_{\fk})\bigr)\otimes1
+1\otimes\bigl(-\nu\rho_{\fp}([X_{i},X_{j}]_{\fk})\bigr).
\end{equation}
We summarize the above in the following table, which leads to the remark below.
\vspace{2mm}
\begin{center}
\small
\renewcommand{\arraystretch}{1.2}
\begin{NiceTabular}{c|c}[hvlines]
$\ {P=G\x_K F}$ &  $\ {{F_{A_\bullet}}}$ at the base point
\\

$\ {G\x_{\pi_{\nu}}\U(E_\nu)}$  & $-\ {\sum_{i<j}^{2n}1\ox X_i^*\w  X_j^*\ox \nu\rho_\fp([X_i,X_j]_\fk)}$
\\
$\ {G\x_{\beta^{\nu,(\pm)}}\GL(S^{(\pm)}\ox E_\nu)}$   & $\ {-\sum_{i<j}^{2n}1\ox X_i^*\w  X_j^*\ox \left( sp^{(\pm)} ([X_i,X_j]_{\fk})\ox1+1\ox \nu\rho_\fp([X_i,X_j]_\fk)\right)}$
\end{NiceTabular}
\vspace{2mm}
\label{AFA}
\end{center}
\begin{remark}\label{rem:trace}
The curvature of the \emph{virtual} connection on $\mathcal{E}_{\nu}$ is the trace of one living on a bundle that always exists. Indeed, for $n\ne1$, \cref{sptraceless} and \cref{eq:Ftwisted} give
\[\frac{1}{m^{(\pm)}}\tr\bigl(F_{A_{S^{(\pm)}\ox E_{\nu}}}\bigr)=-\sum_{i<j}^{2n}1\ox X^{*}_{i}\w X^{*}_{j}\ox\nu\rho_{\fp}([X_{i},X_{j}]_{\fk})
=F_{A_{{E}_{\nu}}} .
\]
For $n=1$ the same holds for $S$, with $\nu$ replaced by $\nu\mp1$ on $S^{\pm}$.
\end{remark}
The following identity is the computational heart of the paper. The proportionality between the curvature and the Kähler form is crucial, since every notion of self-duality containing $\omega$ in its positive eigenspace forces the anti-self-dual part of $F$ to vanish. This property is intended to replace the consequences arising from the isomorphism between self-dual forms and skew-Hermitian endomorphisms of the positive spinors.
\begin{lemma}\label{lem:Fomega}
Let $E_{\nu}$ be a one-dimensional admissible twist and $A_{E_{\nu}}\in\mathcal{A}^G$. Then
\[
F_{A_{E_{\nu}}}=\frac{\nu i}{4\lambda}\,\omega .
\]
Consequently $\Lambda_{\omega}F_{A_{E_{\nu}}}=\tfrac{n\nu i}{4\lambda}$ is constant.
\end{lemma}
\begin{proof}
By the table above and \cref{omegavarpirel},
$F_{A_{E_{\nu}}}=\tfrac{\nu i}{4\lambda}\sum_{i<j}\omega(X_{i},X_{j})X^{*}_{i}\wedge X^{*}_{j}
=\tfrac{\nu i}{4\lambda}\omega$. The last assertion follows from
$\Lambda_{\omega}\omega=|\omega|^{2}=n$.
\end{proof}
\subsection{Equivariant Dirac operators and the Lichnerowicz formula}\label{subsec:dirac}
The definition of the spin representation can be carried out for the reductive case. We consider a reductive $G/H$ and an admissible twist ${(E_\mu,\pi_\mu})$ on $G/H$, equipped with a Hermitian product $h$. If $\tau$ denotes the conjugation in $\fq_{\C}$ with respect to $\fq$, we consider the induced product $(\cdot,\cdot)$ of $(\cdot,\tau(\cdot))_\lambda$ with $h$. We select an orthonormal basis $\{u_k\ox e_l\}$ of $S\ox E_\mu$ with respect to $(\cdot,\cdot)$ and consider a section of the spinor bundle $\psi=\sum_{k,l}\psi_{kl}\otimes u_{k}\otimes e_{l}$. If $\mathrm{st}$ denotes the standard action of $\mathfrak{f}=\mathfrak{gl}(S\ox E_{\mu})$, the evaluation at $\psi$ of a Dirac operator, associated with an arbitrary spin connection $A=A^{ij}\ox X_i^*\ox W_j$ is given below.
\begin{equation}\label{eq:diracgeneric}
\sum_{i=1}^{2n}\sum_{k,l}\bigl[X_{i}(\psi_{kl})\bigr](g)\otimes\gamma(X^{*}_{i})u_{k}\ox e_{l}
+\sum_{i=1}^{2n}\sum_{k,l}\sum_{j=1}^{m^{2}}A^{ij}(g,I)\,\psi_{kl}(g)\otimes
\bigl[(\gamma(X^{*}_{i})\otimes1)\circ\mathrm{st}(W_{j})\bigr](u_{k}\otimes e_{l}).
\end{equation}

\medskip
Since the space of connection one-forms is an affine space over $\mathfrak{f}$-valued one-forms, we may write $A=A_0+a$, where $A_0$ is a fixed $G$-invariant connection and $a\in \Omega^1(M,\mathfrak{f)}$. So, the action of $G$ by left translation commutes with $D_A$ if and only if $a$ is constant on $g$. The $G$-invariance of an element 
    \[a^{ij}\ox X_i^*\ox W_j\in \{C^\infty(G)\ox \fq^*\ox \mathfrak{f}\}^H\]
    forces the components corresponding to the $ {C^{\infty}(G)}$-part to be invariant under the action of $ {H}$. Since the action of $ {H}$ on $\fq^*\ox \mathfrak f$ is not trivial, $ \sum_{i,j}X_i^*\ox W_j\in \{  \fq^*\ox \mathfrak{f}\}^H$. An element in the latter space corresponds to an element of $ {\Hom_{H}(\fq,\mathfrak{f}})$. Thus, we have proved the following. 
\begin{proposition}\label{prop:kernelrep}
If $A\in\mathcal{A}^G$, then $D_A$ is $G$-equivariant and therefore $\ker D_{A}$ is a $G$-representation.
Conversely, if $D_A$ commutes with the left action of $G$, then the calculation
following \cref{eq:diracgeneric} forces $A\in\mathcal{A}^G$.
\end{proposition}

Combining \cref{eq:diracgeneric} with \cref{eq:Aformula}, the Dirac
operator $D^{\Lambda_{E_{\mu}}}_{s}$, associated with
$\Lambda^{s}_{S}\otimes1+1\otimes\Lambda_{E_{\mu}}$ %, viewed inside $\mathcal{U}(\fg)\otimes_{\mathcal{U}(\fk)}\End(S\otimes E_{\mu})$, 
is
\begin{equation}\label{DAsecond}\small
\underset{ {D_0}}{\underbrace{ {\sum_{i}^{2n} r(X_i) \ox \gamma(X_i^*)\ox 1}}} + \underset{ {D_1}}{\underbrace{ {\sum_{i<j<k}^{2n}\tfrac{3}{2}(s[X_i,X_j]_{\fq},X_k)_\fq \ox \gamma(X_i^*)\gamma(X_j^*)\gamma(X_k^*) \ox 1}}}+ \underset{ {D_2}}{\underbrace{ {\sum_{i}^{2n}1\ox \gamma(X_i^*)\ox \Lambda_{E_\mu}(X_i)}},}
\end{equation}
where $r$ is the right action of the universal enveloping algebra $\mathcal{U}(\fg)$, which extends the action $\fg$ by left invariant vector fields: $X\mapsto\frac{d}{dt}\bigl|_{t=0}R_{\exp(tX)}$, where $R_g(f)(g')=f(g'g)$.
%to an algebra map $\mathcal{U}(\fg)\to \text{Diff}(G)^{L}$.

\begin{remark}\label{rem:D0}
The operator $D_0+D_1$ in \cref{DAsecond} is Kostant's cubic Dirac operator for the pair $(\fg,\fh)$ and $s=1/3$, with cubic term $D_1$, see \cite{Ko99}. On a Hermitian symmetric space, $\fh=\fk$, $\fq=\fp$ and $[\fp,\fp]_\fp=0$ annihilates $D_1$. Also, under the assumption that $E_{\mu}=E_{\nu\rho_\fp}$ is one-dimensional, $D_{2}=0$ by \cref{eq:Lambdavanish}. Hence,
\[D^{\Lambda_{E_{\nu}}}_{s}=D_{0}\] 
for every $s$, and we write $D_{0}$ throughout.
\end{remark}

Apart from the criterion regarding the eligibility of the kernel to be a representation and the formula for the Dirac operator, we provide the formula of its square. Although the square of the Dirac operator on $G/K$ in representation-theoretic terms is known, %Proposition 3.1 
see \cite{Par72} for reference, we give an equivalent formulation which matches Lichnerowicz formula and clarifies the relation of $F_{A_{E_\nu}}$ with $F_A$ of the introduction.
\begin{lemma}\label{Y2}
    Let $ {Y\in \fk}$, $ {\psi\in \{ C^{\infty}(G)\ox S\ox E_\nu\}^K}$.
    \begin{align*}
         {[2r(Y)\ox sp(Y)\ox1]\psi}& {=[-r(Y^2)\ox 1\ox1-1\ox sp(Y^2)\ox1+1\ox1\ox\nu\rho_\fp(Y^2)]\psi}\\
         {[2 (1\ox sp(Y)\ox\nu\rho_\fp(Y))]\psi}& {=[-1\ox sp(Y^2)\ox1-1\ox 1\ox\nu\rho_\fp(Y^2)+r(Y^2)\ox1\ox1]\psi}\\
         {[2r(Y)\ox1\ox\nu\rho_\fp(Y)]\psi}& {=[-r(Y^2)\ox 1\ox1-1\ox 1\ox\nu\rho_\fp(Y^2)+1\ox sp(Y^2)\ox1]\psi}
    \end{align*}
\end{lemma}
\begin{proof}
    We write $a=r(Y)\ox1\ox1$, $b=1\ox\sprep(Y)\ox1$ and $c=1\ox1\ox\nu\rho_{\fp}(Y)$ and we need to show 
\[
2ab\,\psi=(c^{2}-a^{2}-b^{2})\psi,\qquad
2bc\,\psi=(a^{2}-b^{2}-c^{2})\psi,\qquad
2ac\,\psi=(b^{2}-a^{2}-c^{2})\psi .
\]
$K$-invariance of $\psi$ reads exactly $(a+b+c)\psi=0$. Then $(a+b)^{2}\psi=c^{2}\psi$, and expanding gives the first identity. The other two follow by symmetry.
\end{proof}
We recall that the Casimir operators $\Omega_\fk$, $\Omega_\fg$, as elements of the universal enveloping algebras $\mathcal{U}(\fk)$, $\mathcal{U}(\fg)$, are expressed as $\Omega_\fk=\sum_i -Y_i^2$ and $\Omega_\fg=\Omega_\fk+\sum_i X_i^2$, where $\{Y_{r}\}$ is a basis of $\fk$ with $(Y_{i},Y_{j})_{\lambda}=-\delta_{ij}$ and $\{X_{i}\}$ an orthonormal basis of $\fp$ with respect to $(\cdot,\cdot)_{\fp,\lambda}$. Their difference $\Omega_\fp\coloneqq \Omega_\fg-\Omega_\fk$ plays the role of the Laplacian in the homogeneous setting.
\begin{proposition}\label{Dsquare}
    The square of the operator $ {D_0}$ applied on $ {\psi\in \{ C^{\infty}(G)\ox S\ox E_\nu\}^K}$ is given by 
     \[ {D_0^2 \psi=\left[-r(\Omega_{\fp})\ox 1\ox 1 +\frac{\scal(\lambda)}{4}\cdot 1\ox 1\ox1-\sum_{1\leq i<j \leq 2n}1\ox \gamma(X_i^*)\gamma(X_j^*)\ox \nu\rho_\fp([X_i,X_j])\right]\psi}.\]
and the last term is $\gamma\bigl(F_{A_{\mathcal{E}_{\nu}}}\bigr)$.
\end{proposition}

\begin{proof}
We prove the case of $G$ non-compact; the compact case is handled by analogous
arguments. Since $\gamma(X^{*}_{i})^{2}=-\mathrm{Id}$ and
$\gamma(X^{*}_{i})\gamma(X^{*}_{j})=-\gamma(X^{*}_{j})\gamma(X^{*}_{i})$,
\[
\Bigl(\sum_{i}r(X_{i})\otimes\gamma(X^{*}_{i})\otimes1\Bigr)^{2}
=-r(\Omega_{\fp})\otimes1\otimes1
+\underbrace{\sum_{i<j}^{2n}r([X_{i},X_{j}])\otimes\gamma(X^{*}_{i})\gamma(X^{*}_{j})\otimes1}_{(*)} .
\]
Choose a basis $\{Y_{t}\}$ of $\fk$ with $(Y_{i},Y_{j})_{\lambda}=-\delta_{ij}$, so
$[X_{i},X_{j}]=-\sum_{t}([X_{i},X_{j}],Y_{t})_{\lambda}Y_{t}$. By
\cref{spinrepformula} and $\ad$-invariance,
$(*)=-2\sum_{t}r(Y_{t})\otimes\sprep(Y_{t})\otimes1$. Applying the first identity of
\cref{Y2} and then the second,
\[
(*)=-2\bigl(1\otimes\sprep(\Omega_{\fk})\otimes1\bigr)
-\sum_{i<j}^{2n}1\otimes\gamma(X_i^{*})\gamma(X_j^{*})\ox\nu\rho_{\fp}([X_{i},X_{j}]).
\]
Finally, since $\gamma(X_i^{*})\gamma(X_j^{*})\gamma(X_i^{*})\gamma(X_j^{*})=-1$
for $i\neq j$, the image of the Casimir simplifies to
\[\sprep(\Omega_{\fk})=-\frac{1}{4}\sum_{i<j}\bigl([X_{i},X_{j}],[X_{i},X_{j}]\bigr)_{\lambda},
\]
which by \cref{eq:scal2} equals $-\tfrac{1}{8}\scal(\lambda)$, hence
$-2\sprep(\Omega_{\fk})=\tfrac{\scal(\lambda)}{4}$. The identification of the remaining term with $\gamma(F_{A_{\mathcal{E}_{\nu}}})$ is the first row of the table in \cref{subsec:curvature}.
\end{proof}
For Hermitian symmetric spaces the simplifications of \cref{rem:D0} give $%\bigl(\nabla^{\Lambda^{s}_{\mathfrak{so}}}\bigr)^{*}\nabla^{\Lambda^{s}_{\mathfrak{so}}}=\bigl(\nabla^{\Lambda^{s}_{S}\otimes\Lambda_{E_{\nu}}}\bigr)^{*}
{\nabla^{\Lambda^{s}_{S}\otimes\Lambda_{E_{\nu}}}}^*\nabla^{\Lambda^{s}_{S}\otimes\Lambda_{E_{\nu}}}=
-r(\Omega_{\fp})\otimes1\otimes1$. Writing $D_{0}=\begin{psmallmatrix}0&D^{-}_{0}\\ D^{+}_{0}&0\end{psmallmatrix}$ with
$D^{-}_{0}={D_0^+}^*$ relative to
$\Gamma(S^{+}\ox\mathcal{E}_{\nu})\op\Gamma(S^{-}\ox\mathcal{E}_{\nu})$ we get the desired identity.
\begin{corollary}(Lichnerowicz formula)\label{Dsquare+} For a four-dimensional Hermitian symmetric space, 
\begin{align}
    {D_0^2\big|\{ C^{\infty}(G)\ox S^{+}\ox E_\nu\}^K}&={D_0^+}^*{{D_0}^+} =-{r\left(\Omega_\fp \right)\ox 1\ox 1+\tfrac{\scal}{4}\cdot 1\ox 1\ox1+\bg({F_{A^{E_\nu}}}^+}).
\end{align} 
\end{corollary}
\subsection{Reducible solutions}
There is a subtle point about \emph{reducible} solutions of the Seiberg--Witten equations, that is, solutions whose spinor is zero. In the unperturbed case, these solutions consist of the zero spinor together with flat connections, the existence of which depends on the admissible parameter $\nu$. Consequently, it cannot be taken for granted and must instead be characterised by the following criterion.
\begin{proposition}\label{prop:flat}
Let $\mathcal{A}^{\mathrm{flat}}\subseteq\mathcal{A}^G$ denote the invariant flat connections on $\mathcal{E}_{\nu}$. Then $\mathcal{A}^{\mathrm{flat}}\neq\varnothing$ if and only if $\nu=0$. The value $\nu=0$ is admissible always in types E\txs{III} and E\txs{VII}, and in the remaining cases, the following criterion must be satisfied.
    \[\begin{NiceTabular}{c|c|c|c|}[hvlines]
    \renewcommand{\arraystretch}{1.3}
A\txs{III} & BD\txs{I} & D\txs{III} & C\txs{I}  \\
${(p+q)/2\in \Z}$ & ${p/2\in \Z}$ & ${(p-1)/2\in \Z}$ & ${(p+1)/2\in \Z}$ 
\end{NiceTabular}\] 
\end{proposition}
\begin{proof}
By \cref{lem:Fomega}, $F_{A_{E_{\nu}}}=\tfrac{\nu i}{4\lambda}\omega$ vanishes if and only if $\nu=0$. Admissibility of $\nu=0$ means $-\rho_\fp$ satisfies the criterion of \cref{LiftTable}, which gives the stated conditions.
\end{proof}

\section{Self-duality in higher dimensions}

Two features of dimension four have to be replaced before the equations can be written down in general: the involutivity of the Hodge operator on two-forms, and
the isomorphism $\Lambda^{2,+}\simeq\fsu(S^{+})$. This section supplies a substitute for the first, shows precisely how the second fails and how this failure is remedied by \cref{lem:Fomega}. The generalised notion of self-duality contributes both to the quadratic term $\sigma(\psi)$ and the variational argument of \cref{varsec}.

\subsection{The twisted Hodge star operator}\label{twistedhodgetable}

Recall that in dimension four the Hodge star operator $*$ is an involution on the space of 2-forms and the space of self-dual forms, $\Lambda^{2,+}$, is exactly its positive eigenspace. Specifically, for the orthonormal basis $\{X^{*}_{i}\}$ of $\fp^*$ for which $X_{i+1}=\ad(Z_0)X_i$, for all odd $i$,
\[
\Lambda^{2,+}=\bigl\langle
\underbrace{X^{*}_{1}\wedge X^{*}_{2}+X^{*}_{3}\wedge X^{*}_{4}}_{\omega},\
\underbrace{X^{*}_{1}\wedge X^{*}_{3}-X^{*}_{2}\wedge X^{*}_{4}}_{\eta},\
\underbrace{X^{*}_{1}\wedge X^{*}_{4}+X^{*}_{2}\wedge X^{*}_{3}}_{\theta}
\bigr\rangle .
\]
The complex-linear extension of the Hodge operator is characterised by $\beta\w\tau(\eta)=(*\beta,\eta)_{\tau}\Omega$, where $(\cdot,\cdot)_\tau$ denotes the Hermitian product $(\cdot,\tau(\cdot))$, ${\beta \in\Lambda^{k}\fp_\C^{*},\, \eta\in\Lambda^{2n-k}\fp_\C^{*}}$ and $ \Omega=\omega^{n}/n!$ is the volume form, and satisfies the usual identities.

\begin{lemma}\label{lib1}\label{lib2} For ${\beta, \, \beta_1,\, \beta_2 \in\Lambda^{k}\fp_\C^{*}}$,
        \begin{equation*}
        \begin{array}{cccccc}
           { (i)}&{\ \tau (\Omega)=\Omega} & { (ii)}&{ (\tau \beta_1,\tau \beta_2)_{\tau }=\tau(\beta_1,\beta_2)_{\tau }} & { (iii)}&{*\tau (\beta)=\tau (*\beta)}\\
           {(iv)}&{ **\beta=(-1)^{k(2n-k)}\beta} & { (v)}&{\ \beta_1\w *\tau (\beta_2)=( \beta_1,\beta_2)_{\tau} \Omega} &  {( vi)}&{ (*\beta_1,*\beta_2)_{\tau}=(\beta_1,\beta_2)_{\tau}}
           \end{array}
        \end{equation*}
    \end{lemma}
  The conjugation invariance of $\Omega$ is inherited by the Kähler form, while both forms are likewise K-invariant. The operator we introduce next is canonical. In the language of higher-dimensional Gauge theory, it coincides with those introduced in \cite{CDFN83, GT00} for manifolds of dimension eight or less. We first construct an invariant $(2n-4)$-form $\Phi$ using the root vectors defined in \cref{sec:spin}.
 \begin{proposition}\label{Phi}
Let $G/K$ be a Hermitian symmetric space and of complex dimension $n\geq 2$. 
\[\Phi:=\frac{\omega^{n-2}}{(n-2)!} \quad \text{ admits the expression }\quad\Phi=i^{\,n-2}\sum_{1\leq i<j\leq n}
\bigwedge_{\substack{1\leq k\leq n\\ k\notin\{i,j\}}}
\left(X_{\alpha_k}^{*}\wedge X_{-\alpha_k}^{*}\right)
%\Phi=i^{\,n-2}\sum_{i<j}^{n}X^{*}_{\alpha_{1}}\w X^{*}_{-\alpha_{1}}\w\dots\w\widehat{X^{*}_{\alpha_{i}}\wedge X^{*}_{-\alpha_{i}}}\w\dots\w\widehat{X^{*}_{\alpha_{j}}\wedge X^{*}_{-\alpha_{j}}}\w\dots\w X^{*}_{\alpha_{n}}\wedge X^{*}_{-\alpha_{n}} .
\]
and it is a $K$-invariant $(2n-4)$-form.
\end{proposition}
\begin{proof}
    The $K$-invariance follows from the invariance of $\omega$. The formula follows by direct computation.
\end{proof}
   % \begin{proof}We set $\omega_k\coloneqq X_{\alpha_k}^{*}\wedge X_{-\alpha_k}^{*}$ and note that $\omega_k^2=0$. We consider $I\subseteq\{1,\ldots,n\}$ and prove the following statement by induction on $r$\[\left(\sum_{k=1}^{n}\omega_k\right)^r =r!\sum_{\substack{|I|=r}}\bigwedge_{k\in I}\omega_k, \qquad 0\leq r\leq n. \]The case $r=0$ is immediate, with the empty wedgeunderstood to be $1$. If the identity holds for $r<n$,multiplication by $\sum_{k=1}^{n}\omega_k$ gives\[\left(\sum_{k=1}^{n}\omega_k\right)^{r+1} =r!\sum_{|I|=r}\sum_{j\notin I}\omega_j\wedge\bigwedge_{k\in I}\omega_k  =(r+1)!\sum_{|J|=r+1}\bigwedge_{k\in J}\omega_k,\]since each subset $J$ occurs exactly $r+1$ times.Taking $r=n-2$ and using $\omega=i\sum_{k=1}^{n}\omega_k$yields the stated expression. The $K$-invariance of$\Phi$ follows from that of $\omega$.   \end{proof}
 The exterior multiplication by $\Phi$ maps a $k$-form to a $(2n-4+k)$-form, to which then we apply the Hodge star operator. The result is a $(4-k)$-form, therefore their composition preserves $\Lambda^2\fp_\C^*$.
   \begin{definition}\label{def:twistedhodge}
Let $G/K$ be a Hermitian symmetric space and $\Phi$ the $K$-invariant $(2n-4)$-form of \cref{Phi}. The \emph{twisted Hodge operator} is
\[*_{\Phi}\colon\Lambda^{2}\fp^{*}_{\C}\to\Lambda^{2}\fp^{*}_{\C}, \qquad \qquad *_{\Phi}\coloneqq *\circ\epsilon(\Phi).\]
\end{definition}
 \begin{lemma}\label{lib4} The twisted Hodge star operator is a $ {K}$-equivariant and $\tau$-equivariant operator.\end{lemma}
      \begin{proof}
         Immediate from the $ {K},\, \tau$-invariance of $\Phi$ and Lemma \ref{lib2}
    \end{proof}
    Even though the twisted Hodge star operator is compatible with the structure, it is not an involutive operator. The eigenspace decomposition splits the space of two-forms into three subspaces.
    \begin{proposition}\label{prop:eigen}
        The twisted Hodge star operator on a Hermitian symmetric space $G/K$ is an involution on $\Lambda^2\fp_\C^*\cap \omega^{\bot}$ and its eigenspace decomposition is 
\[ \Lambda^{2}\fp^{*}_{\C}=V_{+}\oplus V_{-}\oplus\C\omega,\qquad \text{where
}\quad V_{+}=\Lambda^{2}\fp_{\C}^{+\,*}\op\Lambda^{2}\fp_{\C}^{-\,*},\qquad
V_{-}=\bigl(\fp_\C^{+\,*}\w\fp_\C^{-\,*}\bigr)\cap(\C\omega)^{\perp},
\]
The eigenvalues are $+1$, $-1$ and $n-1$ respectively.
\end{proposition}
    \begin{proof}
        %For a unit $ {\eta\in \Lambda^{n-2}\fp_\C^*}$, $\eta\w\tau(*\eta)=\Omega$. So $\tau(*\eta)$ is the sum of complementary with respect to $\Omega$ terms of $\eta$, up to a sign that rests on permutations. The latter depends on the bidegree of $\eta$, namely we consider the following cases.
        Since $X^{*}_{-\alpha_{j}}\wedge X_{\alpha_{j}}^{*}=-X_{\alpha_{j}}^{*}\wedge X^{*}_{-\alpha_{j}}$,
for $\beta=X_{\alpha_{i}}^{*}\wedge X^{*}_{-\alpha_{i}}$ we have $*_{\Phi}\beta=*(\Phi\wedge\beta)=\sum_{j\ne i}X_{\alpha_{j}}^{*}\wedge X^{*}_{-\alpha_{j}}$. Hence, for $\omega=i\sum_{i}X_{\alpha_{i}}^{*}\wedge X^{*}_{-\alpha_{i}}$, we have 
$*_{\Phi}\omega=(n-1)\omega$, while for $i\ne j$,
\[
*_{\Phi}\bigl(X_{\alpha_{i}}^{*}\wedge X^{*}_{-\alpha_{i}}-X_{\alpha_{j}}^{*}\wedge X^{*}_{-\alpha_{j}}\bigr)
=-\bigl(X_{\alpha_{i}}^{*}\wedge X^{*}_{-\alpha_{i}}-X_{\alpha_{j}}^{*}\wedge X^{*}_{-\alpha_{j}}\bigr),
\]
which identifies $V_{-}$. The computation on $V_{+}$ is identical.%, so we result in the following.   \[*_\Phi=\begin{cases}\id&\text{on }\Lambda^{2,0}\fp^*\oplus\Lambda^{0,2}\fp^*,\\-\id&\text{on }\Lambda^{1,1}\fp^*\cap(\C\omega)^\perp,\\(n-1)\id&\text{on }\C\omega.\end{cases}\]
    \end{proof} 
    We note that the eigenspace decomposition of \cref{prop:eigen} is not a decomposition into irreducible representations. In fact, writing $\Lambda^{2}\fp_\C^{*}$ as a direct sum of irreducibles $\bigoplus_i V_{i}$ %with $V_{0}=\C\omega=(\Lambda^{2}\fp_\C^{*})^{K}$,
and noting that $\fp^{\pm}_{\C}$ are irreducible and mutually dual, $K$-equivariance and Schur's lemma make
$*_{\Phi}$ act by a scalar on each summand. 
\medskip

We next define a generalised notion of self-duality by selecting the two-forms which have positive eigenvalue. This recovers the original notion of self-duality in dimension four.
     \begin{definition}\label{selfduality}
Let $G/K$ be a Hermitian symmetric space of complex dimension $n$. For $n\ge2$, a two-form is \emph{self-dual} (resp.\ \emph{anti-self-dual}) if it belongs in the sum of positive (resp.\ negative) eigenspaces of $*_{\Phi}$. Namely,
\[\Lambda^{2,+}\fp_\C^{*}=V_{+}\op\C\omega \qquad \text{ and }\qquad \Lambda^{2,-}\fp_\C^{*}=V_{-}.\] 
For $n=1$, we set $\Phi\coloneqq 0$ and declare every two-form to be self-dual.
\end{definition}
\begin{example}\label{ex:selfdualsmall}
For $n=1$, the volume form coincides with $\omega$, $*_{\Phi}=0$, both eigenspaces are trivial and $\Lambda^{2}\fp_\C^{*}=\C\omega$. For $n=2$, $\Phi$ is
the $0$-form $1$, so $*_{\Phi}=*$ and the two positive eigenvalues coincide, recovering the classical $\Lambda^{2,\pm}$. For $n>2$ one has $\dim\Lambda^{2,+}=n^{2}-n+1<n^{2}-1=\dim\Lambda^{2,-}.$
\end{example}
    Due to the $K$-invariance of the twisted Hodge operator, the notion of self-dual projection is independent of the choice of basis. We fix the following real orthonormal basis $\{ \omega/\sqrt{n}, \eta_{ij}, \theta_{ij} \}_{1\leq i<j\leq n}$ of self-dual elements.
    \begin{gather}\label{selfdualelements}
\begin{gathered}
  \frac{\omega}{\sqrt n}
=\frac{X^{*}_{1}\wedge X^{*}_{2}+\dots+X^{*}_{2n-1}\wedge X^{*}_{2n}}{\sqrt n}, \\[1ex]
  \eta_{ij}\coloneqq\frac{X^{*}_{2i-1}\wedge X^{*}_{2j-1}-X^{*}_{2i}\wedge X^{*}_{2j}}{\sqrt2},
  \qquad\qquad
  \theta_{ij}\coloneqq\frac{X^{*}_{2i-1}\wedge X^{*}_{2j}+X^{*}_{2i}\wedge X^{*}_{2j-1}}{\sqrt2}.
\end{gathered}
\end{gather}
%    \begin{gather}\frac{\omega}{\sqrt n}=\frac{X^{*}_{1}\wedge X^{*}_{2}+\dots+X^{*}_{2n-1}\wedge X^{*}_{2n}}{\sqrt n},\\[2pt] \eta_{ij}\coloneqq\frac{X^{*}_{2i-1}\wedge X^{*}_{2j-1}-X^{*}_{2i}\wedge X^{*}_{2j}}{\sqrt2}, \qquad \theta_{ij}\coloneqq\frac{X^{*}_{2i-1}\wedge X^{*}_{2j}+X^{*}_{2i}\wedge X^{*}_{2j-1}}{\sqrt2}. \end{gather}

\begin{definition}\label{def:proj}
Let $G/K$ be a Hermitian symmetric space of complex dimension $n$. For the self-dual elements fixed in \cref{selfdualelements}, the \emph{self-dual projection} is
\[\Pi^{+}:=\dfrac{\Lambda_{\omega}}{n}\,\omega
+\sum_{i<j}^{n}\bigl(\Lambda_{\eta_{ij}}\eta_{ij}+\Lambda_{\theta_{ij}}\theta_{ij}\bigr),\]
where $\Lambda_{\varphi}(\cdot)=(\varphi,\cdot)_{\lambda}$ denotes contraction.
For $n=2$, this is the classical $\Pi^{+}=(1+*)/2$.
\end{definition}
     The generalised notion of self-duality does not recover the natural isomorphism ${\gamma:\Lambda^{2,+}\fp_\C^*\rightarrow\End_0(S^+)}$ of the $n=2$ case for reasons of dimension alone.
     \begin{center}
     \vspace{2mm}
     \small
     \renewcommand{\arraystretch}{1.3}
        \begin{tabular}{|c|c|c|c|c|c|c|c|c|}\hline
            $\dim_\C\fp^*$ &${n}$& 1 & 2 & 3 & 4 & 5 & 6 & 7\\
            \hline
            %2. & $\dim_\C\fp_\C^* $&${2n}$& 2 & 4 & 6 & 8 & 10 & 12 & 14\\ \hline
            ${\dim_\C\Lambda^2\fp_\C^*} $ &${\binom{2n}{2}}$&  1 & 6 & 15 & 28 & 45 & 66 & 91\\ \hline
            $\dim_{\C}\Lambda^{2,+}\fp_\C^{*}$ & $n^{2}-n+1$ & $1$ & $3$ & $7$ & $13$ & $21$ & $31$ & $43$\\
            \hline
            %4.& $ {\dim_\C S^+}$& ${2^{n-1}}$ & 1 & 2 & 4 & 8 & 16 & 32 & 64\\ \hline
            %5. & ${\dim_\C\End(S^+) }$ &${2^{2(n-1)}}$ & 1 & 4 & 16 & 64 & 256 & 1024 & 4096\\ \hline
            ${\dim_\C\End_0(S^+) }$ &${4^{(n-1)}-1}$ & 0 & 3 & 15 & 63 & 255 & 1023 & 4095
            \\
            \hline
            %7. &${\dim_\C S}$ &${2^{n}}$ & 2 & 4 & 8 & 16 & 32 & 64 & 128 \\ \hline
            
            %8. & ${\dim_\C\End(S)}$ &${2^{2n}}$ & 4 & 16 & 64 & 256 & 1024 & 4096 & 16384\\
            
            \hline
        \end{tabular}
        \vspace{2mm}
        \end{center}
        The third and fourth rows agree only at $n=2$, while for $n>3$ the dimensions in the second and third rows are always smaller than those in the fourth. Thus, the four-dimensional identity $\Lambda^{2,-}\fp^{*}_{\C}=\bigcap_{\psi\in S^{+}}\ker\gamma(\cdot)\psi$  induced by $\Lambda^{2,+}\simeq\fsu(S^{+})$ fails for $n>2$, the intersection being trivial. Regarding curvature, this does not obstruct the generalisation of \cref{Dsquare+} in higher dimensions since ${F_{A^{ E_\nu}}^-=0}$ from \cref{lem:Fomega} 
    \begin{align*}
    {D_0^2\big|\{ C^{\infty}(G)\ox S^{+}\ox E_\nu\}^K}&=-{r\left(\Omega_\fp \right)\ox 1\ox 1+\tfrac{\scal}{4}\cdot 1\ox 1\ox1+\gamma({F_{A^{E_\nu}}}^+}).
    \end{align*}

\subsection{The norm identity and the squaring map}

Next we examine how the squaring map generalises to higher dimensions and the role self-duality plays in this natural choice. We first record the action of the Kähler form and isolate a genuinely exceptional feature of dimension four. 

\begin{lemma}\label{lem:cliffordomega}
The Clifford action of the Kähler form satisfies
${\gamma(\omega)}_{{|}_{\Lambda^{k}\fp_\C^{-\,*}}}=i(2k-n)\,\mathrm{id}$. In particular,
${\gamma(\omega)}_{{|}_{\Lambda^{2k}\fp_\C^{-\,*}}}=i(4k-n)\,\mathrm{id}$ for $0\le k\le\lfloor n/2\rfloor$.
\end{lemma}
\begin{proof}
    \begin{align*}
\gamma(\omega)
&=\gamma\left(i\sum_{j=1}^n X_{\ga_j}^*\w X_{-\ga_j}^*\right)=
i
\sum_{j=1}^n(\gamma(X_{\alpha_j}^*)\gamma(X_{-\alpha_j}^*)+\id)=i
\left(\sum_{j\in I}\id-\sum_{j\not\in I}\id\right)=i(2k-n)\id.
\end{align*}
\end{proof}
\begin{proposition}\label{prop:norm}
Let $G/K$ be of complex dimension $n$ and let
$\zeta\in\Lambda^{2,+}(G/K,i\R)$. Then,
\[\bigl|\gamma_{{|}_{S^{+}}}(\zeta)\bigr|^{2}=a_n|\zeta|^{2},\qquad \text{where} \qquad 
a_n=
\begin{cases}
2^{\,n-1}, & n\ne2,\\[2pt]
2^n, & n=2 .
\end{cases}
\]
\end{proposition}
\begin{proof}
We work in the orthonormal basis $\{\omega/\sqrt n,\eta_{ij},\theta_{ij}\}$. Since $\gamma(a)^{*}=-\gamma(a)$, \cref{lem:cliffordomega} gives
\[
\Bigl|\gamma_{{|}_{S^{+}}}\Bigl(\tfrac{\omega}{\sqrt n}\Bigr)\Bigr|^{2} =-\frac{1}{n}\tr\bigl(\gamma_{{|}_{S^{+}}}(\omega)^{2}\bigr)
=\frac{1}{n}\sum_{\substack{k=0\\ k\ \mathrm{even}}}^{n}(2k-n)^{2}\binom nk=\begin{cases}
    2^{n-1}, & n\ne 2
    \\
    2^n, &n=2.
\end{cases} 
\]
For the remaining basis elements,
$\gamma(\eta_{ij})^{2}=-1-\gamma(X^{*}_{2i-1})\gamma(X^{*}_{2j-1})\gamma(X^{*}_{2i})\gamma(X^{*}_{2j})$,
and for $n\ne2$ the second summand is off-diagonal, so
$|{\gamma}_{{|}_{S^{+}}}(\eta_{ij})|^{2}=\dim S^{+}=2^{n-1}$, and the same holds for $\theta_{ij}$. The claim for two-forms follows from orthogonality of \cref{selfdualelements} and this is carried for $i\R$-valued forms as
well. For $n=2$ each of the three basis elements contributes the factor $4$, as a
direct computation from \cref{eq:cliffordn2}.
\end{proof}

We mention a natural approach to the generalised statement of the Seiberg--Witten equations. First, we remind that the $\zeta$-perturbed equations of \cref{sweq4} are $2F^{+}_{A}+2\zeta=-\sigma(\psi)$. The generalised notion of self-duality allows the expression of the left-hand side of the monopole equation to any complex dimension. The quadratic term can be expressed via one of the following approaches
\begin{equation}\label{eq:sigma}
\sigma(\psi)\coloneqq\frac{1}{a_{n}}\sum_{i<j}^{2n}\bigl(\gamma(X^{*}_{i})\gamma(X^{*}_{j})\psi,\psi\bigr) X^{*}_{i}\w X^{*}_{j}, \qquad \mathcal{Q}(\psi)\coloneqq\psi\psi^{*},
\end{equation}
where $\Quad_0(\psi)$ denotes the traceless part of $\Quad(\psi)$. For the original case of $n=2$, using the basis furnished by \cref{uweights} and the elementary matrices $W_{ij}$,
\begin{equation}\label{eq:cliffordn2}
\gamma(\omega)=-2iW_{11}+2iW_{22},\qquad
\gamma(\eta)=-2W_{12}+2W_{21},\qquad
\gamma(\theta)=2iW_{12}+2iW_{21}.
\end{equation}
so one readily verifies that $\gamma(\sigma(\psi))=-\Quad_0(\psi)$. This holds solely for $n=2$ and $n=3$, so for $n>3$ each approach results in a different expression of the monopole. Although the map $Q_0$ satisfies the following properties, its choice for the generalisation of the Seiberg--Witten equations would result in a system with no irreducible solutions for $n\geq4$.
\begin{lemma}\label{lem:Qzero}
If $\dim(S^+\ox E_\nu)=m=2^{n-1}$, then $\bigl(\Quad_0(\psi)\psi,\psi\bigr)
=\frac{m-1}{m}\,|\psi|^{4}
=\bigl\langle\Quad_0(\psi),\Quad_0(\psi)\bigr\rangle.$
\end{lemma}
\begin{proof}
Applying $\Quad_0(\psi)=\psi\psi^*-\frac{1}{m}|\psi|^2\id$ to $\psi$ we get $\Quad_0(\psi)\psi=\left(1-\frac{1}{m}\right)|\psi|^2\psi$. Thus $(\Quad_0(\psi)\psi,\psi) =\frac{m-1}{m}|\psi|^4$. Also, $
\tr(\Quad_0(\psi)^2)=\tr\left((\psi\psi^*)^2-\frac{2|\psi|^2}{m}\psi\psi^*+\frac{|\psi|^4}{m^2}\id\right)=|\psi|^4-\frac{2}{m}|\psi|^4+\frac{1}{m}|\psi|^4=\frac{m-1}{m}|\psi|^4.$
\end{proof} 
\begin{proposition}\label{Qnot}
If $n\geq 4$ and $c$ is a constant, the condition $\Quad_0(\psi)=ic \gamma(\omega)$ forces $\psi$ to be $0$.
\end{proposition}

\begin{proof}
The condition reads $\psi\psi^{*}\,=\,\frac{1}{m}|\psi|^{2}\,\id\,+\,c\,i\gamma(\omega)$. By \cref{lem:cliffordomega}, the right-hand side is diagonal for the decomposition of $S^{+}\ox E_\nu$ by degree, %$=\bigoplus_{k}\Lambda^{2k}\fp_\C^{-*}$, 
with eigenvalue
\[\frac{1}{m}|\psi|^{2}-c(4k-n) \quad \text{ of multiplicity }\binom{n}{2k}\text{ on the k-th summand}.\] 
The numbers $4k-n$, $0\le k\le\lfloor
n/2\rfloor$, are pairwise distinct, so the map assigning $k$ to $\tfrac{1}{m}|\psi|^{2}-c(4k-n)$, $c\neq0$,
is injective and at most one of its values vanishes. The left-hand side has rank at most one, so at most one of these eigenvalues is non-zero. Hence
$\lfloor n/2\rfloor+1\le2$, that is, \ $n\le3$. For $n\ge4$, $c=0$ and
$|\psi|=0$.
\end{proof}

%Both of those reasons lead naturally for our choice of normalised map to be $-\gamma\circ\sigma$ of \cref{eq:sigma}. Since it does not satisfy the first condition outside of $n=2$, we adjust the functional accordingly. 

\section{The equivariant Seiberg--Witten equations}\label{sec:gen}

\subsection{Variational origin}\label{varsec}
Energy functionals in Gauge theory are integrals of a kinetic term, a gauge-field term and a potential. Let $M=G/K$ be a compact Hermitian symmetric space of complex dimension $n$, $E_{\nu}$ be a one-dimensional admissible twist and $A\in\mathcal{A}$, where $\mathcal A$ henceforth denotes the space of unitary connections on $\mathcal{L}_\nu$. We consider
\begin{equation}\label{eq:energy}
\mathcal{E}(\psi,A)\coloneqq\int_{M}\left(
\bigl|\nabla_{A}\psi\bigr|^{2}
+\bigl|F_{A}\bigr|^{2}
+\frac{\scal}{4}|\psi|^{2}
+\frac{a_{n}\bigl|\gamma(\sigma(\psi)^{+})\bigr|^{2}}{32}
\right){\vol}_{M}.
\end{equation}
For $n=2$ one has $a_{2}=4$ and, by \cref{lem:Qzero} and
$\gamma(\sigma(\psi))=-\Quad_0(\psi)$, the last term is $\tfrac1{16}|\psi|^{4}$. This simplifies \cref{eq:energy} to the energy functional \cref{eq:sw-energy-intro} of the introduction. 

\begin{proposition}\label{prop:topological}
Let $M=G/K$ be compact of complex dimension $n\ge2$, $E_{\nu}$ a one-dimensional
admissible twist and $A\in\mathcal{A}$. Then
\[
\int_{M}\Bigl(-\bigl|F^{+}_{A}\bigr|^{2}+\bigl|F^{-}_{A}\bigr|^{2}\Bigr){\vol}_{M}
+c\,\bigl\langle c_{1}(\mathcal{L}_{\nu})^{2}\smile[\omega]^{n-2},[M]\bigr\rangle
\ \ge\ \frac{(n-2)}{n\,\vol(M)}\biggl|\int_{M}\Lambda_{\omega}F_{A}\,{\vol}_{M}\biggr|^{2},
\]
where $c=4\pi^{2}/(n-2)!$, with equality, for $n>2$, if and only if
$\Lambda_{\omega}F_{A}$ is constant. For $n=2$, the formula is the classical topological identity and equality holds for every $A\in\mathcal{A}$.
\end{proposition}

\begin{proof}
By \cref{def:twistedhodge} and \cref{lib1},
$(*_{\Phi}F_{A},F_{A})_{\tau}\Omega=F_{A}\wedge\Phi\wedge\tau(F_{A})$, and since $\Phi$ is closed the integral of this expression is independent of $A$, hence a multiple of
$\langle c_{1}(\mathcal{L}_{\nu})^{2}\smile[\omega]^{n-2},[M]\rangle$, where $\mathcal{L}_{\nu}=\mathcal{E}_{\nu}\ox \mathcal{E}_{\nu}$% (see \cref{2nu})
. We write $F$ for $F_{A}$. Setting $F_{\omega}\coloneqq (\Lambda_{\omega}F)\omega/n$ and decomposing $\Pi^{+}F=F_{\omega}+F_{{|}_{V_+}}$, $F_{{|}_{V_-}}:=F-F_{\omega}-F_{{|}_{V_+}}$, \cref{prop:eigen} gives
\[
(*_{\Phi}F,F)_\tau=|F_{{|}_{V_+}}|^{2}-|F_{{|}_{V_-}}|^{2}+(n-1)|F_{\omega}|^{2}
=|F^{+}|^{2}-|F^{-}|^{2}+(n-2)|F_{\omega}|^{2}.
\]
Since $|F_{\omega}|^{2}=|\Lambda_{\omega}F|^{2}/n$, the Cauchy--Schwarz inequality
$\bigl|\int_{M}\Lambda_{\omega}F\ \vol_M\bigr|^{2}\le n\vol(M)\int_{M}|F_{\omega}|^{2}\ \vol_M$
holds, with equality if and only if $\Lambda_{\omega}F$ is constant. Rearranging the terms gives the claim.
\end{proof}
If $\mathcal{C}_\nu=\Gamma(\mathcal{S^+\ox E}_\nu)\x \mathcal{A}$, we denote by $\mathcal{C}_{SD}$ the space consisting of pairs $(\psi,A)\in \mathcal{C}_\nu$ such that the Hermitian--Yang--Mills type condition $F_{A}^-=0$ is satisfied. 
\begin{theorem}\label{theo:var}
Let $M=G/K$ be compact of complex dimension $n$, let $E_{\nu}$ be a one-dimensional admissible twist and let $A\in\mathcal{A}$. A pair $(\psi,A)\in \mathcal{C}_{SD}$ is an absolute minimum of the energy functional $\mathcal{E}_{{|}_{\mathcal{C}_{SD}}}$ if
\begin{equation}\label{swgen}
D_{A}^+\psi=0,
\qquad
F^{+}_{A}=-\frac{a_{n}}{8}\,\sigma(\psi)^{+},\qquad \Lambda_{\omega}F_{A}\text{ is constant.}
\end{equation}
 The last condition is vacuous for $n=2$, and is automatic for every $A\in\mathcal{A}^G$ by \cref{lem:Fomega}.
\end{theorem}

\begin{proof}
We write $F$ for $F_{A}$. By the generalised Lichnerowicz formula in \cite{LM89} and the fact that on $\mathcal{C}_{SD}$ we have $F^-=0$, the first and third terms of \cref{eq:energy} combine into
\begin{equation}
\mathcal{E}(\psi,A)\coloneqq\int_{M}\left(
\bigl|D_{A}^+\psi\bigr|^{2}-\frac{1}{2}\bigl(\gamma(F^{+})\psi,\psi\bigr)
+\bigl|\,F\,\bigr|^{2}
+\frac{a_{n}\bigl|\gamma(\sigma(\psi)^{+})\bigr|^{2}}{32}
\right){\vol}_{M}.
\end{equation}
We isolate the square of the Dirac operator and assume that $n> 2$. The remaining terms are 
\begin{equation}\label{presquare}
    -\frac{1}{2}\bigl(\gamma(F^{+})\psi,\psi\bigr)
+\frac{1}{2^{-n+6}}|\,\gamma(\sigma(\psi)^{+})\,|^{2}+|\,F\,|^{2}.
\end{equation}
By \cref{prop:norm}, $|\,\gamma(F^{+})\,|^{2}=2^{n-1}|\,F^{+}\,|^{2}$, and by the definition of $\sigma$ in \cref{eq:sigma},
\[\bigl(\gamma(F^{+})\psi,\psi\bigr)=-\bigl\langle\gamma(F^{+}),\gamma(\sigma(\psi)^{+})\bigr\rangle.\] Completing the square in \cref{presquare}, we get
\[
2^{\,n-6}\Bigl(2^{-n+4}\gamma(F^{+})+\gamma(\sigma(\psi)^{+})\Bigr)^{2}
\  -|\,F^{+}\,|^{2}+|\,F^{-}\,|^{2}.
\]
So, 
\begin{equation}\label{squares}
\mathcal{E}(\psi,A)=
\bigl\|D_{A}\psi\bigr\|_{L^2}^{2}
+2^{n-6}\bigl\|2^{-n+4} \, \gamma(F^+)+\gamma(\sigma(\psi)^+)\bigr\|_{L^2}^{2}+\int_M( -|\,F^{+}\,|^{2}+|\,F^{-}\,|^{2})\vol_M ,
\end{equation}
and the last term is bounded below by \cref{prop:topological}, with equality exactly when $\Lambda_{\omega}F$ is constant. For $n=1$, $\psi$ and $F$ are zero, and the claim follows. For $n=2$, $|\,\gamma(F^{+})\,|^{2}=4|\,F^{+}\,|^{2}$, and the square completes to $F^+=-1/2\sigma(\psi)^+$. The last summand is purely topological and the constancy condition is empty, which is the classical situation.

\end{proof}

\begin{corollary}\label{cor:equivmin}
Let $M=G/K$ be compact. Every pair $(\psi,A)$ with $A\in\mathcal{A}^G$ satisfying
\cref{swgen} is an absolute minimum of $\mathcal{E}_{{|}_{\mathcal{C}_{SD}}}$.
\end{corollary}

Their perturbed version also has a variational nature tied with the energy functional induced after expanding the perturbed squares in \cref{squares}. The compactness hypothesis in \cref{prop:topological}, \cref{theo:var} and \cref{cor:equivmin} is not removable: on a non-compact dual the volume is infinite and $\mathcal{E}$ need not be finite. Nevertheless, we define the equations on both settings, having the discussion above as motivation for the non-compact ones.
\begin{definition}\label{eqSW}
    The $it\omega$-perturbed equations of \cref{swgen} are the \emph{equivariant Seiberg--Witten equations}. 
    \begin{equation}\label{swcngen}
D_{A}^+\psi=0,
\qquad
F^{+}_{A}+it\omega=-\frac{a_{n}}{8}\,\sigma(\psi)^{+}.
\end{equation}
   We denote by $\mathcal{S}_\nu^\infty$ the space of their solutions.
\end{definition}

\subsection{The monopole equation in explicit form}\label{sec:mono}
The following theorem, combined with \cref{theo:var}, recovers Theorem A of the introduction. We use the decomposition of a positive spinor by exterior degree,
\begin{equation*}\label{eq:degdecomp}
\psi=\sum_{k=0}^{\lfloor n/2\rfloor}\psi^{(2k)},
\qquad \psi^{(2k)}\in\Lambda^{2k}\fp_\C^{-\,*}\otimes E_{\nu}.
\end{equation*}

\begin{theorem}\label{theo:monopole}
Let $(G/K,\lambda B_{\fp})$ be of complex dimension $n$, $E_{\nu}$ be a one-dimensional admissible twist and $A\in\mathcal{A}^G$. With $\mono(t,\nu,\lambda)=8n\bigl(t+\tfrac{\nu}{2\lambda}\bigr)$, the $it\omega$-perturbed monopole equation of \cref{swgen} is equivalent to the
pointwise system
\begin{equation*}
\sum_{k=0}^{\lfloor n/2\rfloor}(n-4k)\bigl|\psi^{(2k)}(g)\bigr|^{2}=\mono(t,\nu,\lambda),
\qquad
\bigl(\gamma(\eta_{ij})\psi,\psi\bigr)(g)=\bigl(\gamma(\theta_{ij})\psi,\psi\bigr)(g)=0,
\quad 1\le i<j\le n .
\end{equation*}
\end{theorem}

\begin{proof}
The perturbed $it\omega$-monopole reads $2F_{A_{E_{\nu}}}^{+}+it\omega=-\tfrac{a_n}{8}\sigma(\psi)^{+}$, since \cref{linebndleequiv} indicates that $F_{A}=2F_{A_{E_{\nu}}}$. By \cref{lem:Fomega}, the left-hand side is $\bigl(\tfrac{\nu i}{2\lambda}+it\bigr)\omega$, so all $\eta_{ij}$- and $\theta_{ij}$-components of the left-hand side vanish, which leads to the second and third families of equations. Contracting with $\omega$, the left-hand side contributes $n\bigl(\tfrac{\nu i}{2\lambda}+it\bigr)$, while the right-hand side, by
\cref{eq:sigma} and \cref{lem:cliffordomega}, contributes
\[
-\frac{a_n}{8}\cdot\frac{1}{a_n}\bigl(\gamma(\omega)\psi,\psi\bigr)
=-\frac18\sum_{k}i(4k-n)\bigl|\psi^{(2k)}\bigr|^{2}.
\]
Multiplying by $8/i$ gives the first equation.
\end{proof}

\begin{example}\label{ex:lowdimmono}
Writing $\psi_{1},\dots,\psi_{2^{n-1}}$ for the components in the basis of \cref{uweights}, ordered by increasing degree, \cref{theo:monopole} reads:
\[
\begin{array}{ccccc}
n
& &
\omega\text{-condition}
& &
\eta_{ij}\ \&\ \theta_{ij}\text{-condition}
\\[2pt]
\hline
\\
2 & &
{2|\psi_{1}|^{2}-2|\psi_{2}|^{2}=\mono[2]} & &
{\psi_{1}\overline{\psi_{2}}=0}
\\[10pt]
3 & &
{3|\psi_{1}|^{2}-|\psi_{2}|^{2}-|\psi_{3}|^{2}-|\psi_{4}|^{2}=\mono[3]}
& &
{\psi_{1}\overline{\psi_{2}}}
={\psi_{1}\overline{\psi_{3}}}
={\psi_{1}\overline{\psi_{4}}=0}
\\[10pt]
4 & &
{4|\psi_{1}|^{2}-4|\psi_{8}|^{2}=\mono[4]} & &
{\psi_{1}\overline{\psi_{2}}+\overline{\psi_{7}}\psi_{8}}=0,
\quad 
{\psi_{1}\overline{\psi_{3}}-\overline{\psi_{6}}\psi_{8}}=0,
\\
& & & &{\psi_{1}\overline{\psi_{4}}+\overline{\psi_{5}}\psi_{8}}=0,\quad
{\psi_{1}\overline{\psi_{7}}+\overline{\psi_{2}}\psi_{8}}=0,
\\
& & & &\psi_{1}\overline{\psi_{6}}-\overline{\psi_{3}}\psi_{8}=0,\quad
\psi_{1}\overline{\psi_{5}}+\overline{\psi_{4}}\psi_{8}=0.
\end{array}
\]
\end{example}
The case $n=2$ is the classical monopole. The case $n=4$ exhibits the first genuinely new phenomenon: the six middle components carry coefficient $n-4k=0$ and therefore do not appear in the first equation at all, although they are constrained by the second and third families. Additionally, it is the first case where these families contain sums of products, foreshadowing their complicated expression in higher dimensions. The shape of the monopole condition is captured in the following proposition.
\begin{proposition}[General picture of the monopole]\label{monoshape}
The equations of \cref{theo:monopole} take the form
\[
\sum_{k}d_{k}\bigl|\psi^{(2k)}\bigr|^{2}=\mono(t,\nu,\lambda),
\qquad
\sum_{i_{1}\neq i_{2}}\varepsilon_{i_{1}i_{2}}\,\psi_{i_{1}}\overline{\psi_{i_{2}}}=0,
\qquad d_{k}\in\Z,\ \varepsilon_{i_{1}i_{2}}\in\{0,\pm1\}.
\]
\end{proposition}
\begin{proof}
    We refer to \cite{Tsi26} for the proof.
\end{proof}

\subsection{Gauge group and equivariant solutions}

As in the four-dimensional case, we examine solutions of the ${it\omega}$-perturbed equivariant Seiberg--Witten equations with ${A\in \mathcal{A}^G}$. Since the statement is given for all Hermitian symmetric spaces, regardless of their (non-)compactness, we will restrict to solutions ${(\psi,A)}$ such that the spinor $\psi$ belongs in the $L^2$ kernel of the Dirac operator. We first note that solutions of \cref{swgen} are invariant under the gauge group $\mathcal{G}\simeq  C^{\infty}(G/K,U(1))$, acting by $u\cdot(\psi,A)=(u\psi,A-2u^{-1}du)$. The curvature transforms by conjugation while its trace component is invariant by \cref{rem:trace}, and $\sigma(\psi)$ is preserved by scalar, unitary transformations.

\begin{proposition}\label{prop:gaugeeq}
A gauge transformation $u$ preserves the set of pairs $(\psi,A)$ for which
$A\in\mathcal{A}^G$ if and only if $u\colon G/K\to U(1)$ is constant.
Consequently, the equivariant gauge group is $\mathcal{G}_{eq}\simeq U(1)$.
\end{proposition}
\begin{proof}
Under a gauge transformation $u:G/K\to\U(1)$, the connection changes by the Maurer-Cartan term $-2u^{-1}du$. For $(u\psi,(u^{-1})^*A)$ to remain an equivariant solution, $u^{-1}du$ must be a $G$-invariant $1$-form. Since the action of $K$ on $\fp$ is irreducible, it has no non-zero trivial subrepresentation, therefore $u^{-1}du=0$. So, $du=0$, and since $G/K$ is connected, $u$ is constant. The converse is readily verified.
\end{proof}

\begin{definition}\label{def:eqsolutions}
We denote $\mathcal{M_\nu^\infty}(c_n)$ the moduli space $\mathcal{S}_{\nu}^{\infty}(c_n)/\mathcal{G}$. The sets of \emph{equivariant solutions} are
\[
\mathcal{S}_{\nu,eq}^{\infty}(c_n):=\bigl\{(\psi,A)\ \text{satisfying \cref{swcngen}}\, |\, A\in\mathcal{A}^G\bigr\},
\qquad
\mathcal{S}_{\nu,eq}^{L^{2}}(c_n)\coloneqq\mathcal{S}^{\infty}_{\nu,eq}(c_n)\cap\bigl\{\psi\in\ker_{L^{2}}D^{+}_{0}\bigr\}.
\]
The corresponding moduli sets are
$\mathcal{M}^{\infty}_{\nu,eq}(c_n)\coloneqq\mathcal{S}^{\infty}_{\nu,eq}(c_n)/\mathcal{G}_{eq}$ and
$\mathcal{M}^{L^{2}}_{\nu,eq}(c_n)\coloneqq\mathcal{S}^{L^{2}}_{\nu,eq}(c_n)/\mathcal{G}_{eq}$.
\end{definition}

By ellipticity of the Dirac operator $D_0$, every $L^{2}$ kernel element is smooth, see \cite{Par72} for reference. Thus,
\[
G\ \text{compact}\ \Longrightarrow\ \mathcal{S}_{\nu,eq}^{L^{2}}=\mathcal{S}_{\nu,eq}^{\infty},
\qquad
G\ \text{non-compact}\ \Longrightarrow\ \mathcal{S}_{\nu,eq}^{L^{2}}\subset\mathcal{S}_{\nu,eq}^{\infty}.
\]
We conclude this section with a fact about invariant spinors that will be used repeatedly in the study of equivariant solutions.

\begin{proposition}\label{cor:muzero}
    Let $G/K$ be a compact irreducible Hermitian symmetric space, let $E_{\mu\rho_{\fp}}$ be a one-dimensional $K$-representation and let $\psi\in H^{(0,0)}(G/K,E_{\mu\rho_{\fp}})$ have non-zero constant pointwise norm. Then $\mu=0$.
\end{proposition}

\begin{proof}
Indeed, $\psi$ is then nowhere vanishing and holomorphic, hence trivialises $\mathcal{E}_\mu$ as a holomorphic line bundle, thus $c_1(\mathcal{E}_\mu) = 0$. By \cref{lem:Fomega}, 
\[F_{A_{E\mu}} = \frac{\mu i}{4\lambda}\omega.\] 
Since $[\omega]\neq 0$, this forces $\mu= 0$.
\end{proof}

\section{Geometric realisation of irreducible representations}\label{sec:geomreal}
Two fundamental objectives of representation theory of semisimple Lie groups are the classification of irreducible representations and their geometric realisation. We study the realisation of irreducible representations as kernels of Dirac operators twisted by the one-dimensional admissible twist ${E_{\nu\rho_p}}$. Our aim is to determine explicitly the conditions on $\nu$ under which such a realisation occurs.

\subsection{Compact case}\label{sec:compact}
 By elliptic regularity, the $L^2$ kernel coincides with the smooth kernel of the Dirac operator. According to %Theorems 3 \& 4 of
 \cite{La00}, its restriction to the positive spinor bundle is realised as a highest weight representation, 
\[\ker\, D^+(\mathcal{E_{\mu}})=V_{w(\mu+\rho_\fk)-\rho_\fg},%\qquad \qquad D^+(\mathcal{E_{\mu}}):L^2(\mathcal{S^+\ox E_\mu})\rightarrow L^2(\mathcal{S^-\ox E_\mu})
\] 
    if there exists a Weyl element ${w\in W_G}$ of even length such that ${w(\mu+\rho_\fk)-\rho_\fg}$ is dominant for G; in any other case ${\ker\, D^+(\mathcal{E_{\mu}})=0}$. We provide the Weyl group elements and their length for all compact, irreducible Hermitian symmetric spaces; we first recall that if such an element exists, it is unique.
\begin{remark}\label{uniquew}
    The Weyl group acts simply transitively on the Weyl chambers of the regular elements. Since $\langle\rho_{\fg},\check\alpha_{i}\rangle=1$ for every simple root $\alpha_{i}$, $\langle w(\mu+\rho_{\fk}),\check\alpha_{i}\rangle\ge1>0$, where $\langle\lambda,\check{\alpha}\rangle
 =2(\lambda,\alpha)/(\alpha,\alpha)$. So, $w(\mu+\rho_{\fk})$ is regular dominant and, by transitivity, a regular Weyl orbit meets the open dominant chamber once.
\end{remark}
\begin{lemma}\label{rhoND}
        If $G/K$ is an irreducible Hermitian symmetric space and ${\alpha\in \Delta_{\fp}^+}$, then ${( \rho_{\fp},\alpha )=N}$ if G/K is of Type C\txs{I} and ${( \rho_{\fp},\alpha)=N/2}$ otherwise.
    \end{lemma}
\begin{proof}
    By \cref{Z_0*rho}, $\rho_{\fp}$ is proportional to $Z_0^{*}$, and the ordering of Convention \ref{conventionHSS} gives $\alpha(Z_{0})=i$ for every $\alpha\in\Delta_\fp^{+}$, so $(\rho_{\fp},\alpha)$ is independent of $\alpha$. Its value is read off type by type.
\end{proof}
\begin{convention}\label{conv:weyl}
We set ${s_{i}\coloneqq s_{\gve_i-\gve_{i+1}}}$, ${\tilde{s}_{i}\coloneqq s_{\gve_i+\gve_{i+1}}}$ and ${\check{s}_{p,i} \coloneqq s_{p-1} }$ when $i$ even and ${\check{s}_{p,i} \coloneqq \tilde{s}_{p-1}} $ when ${i}$ odd. Products of reflections are composed left to right, that is, $s_{i}s_{j}$ means that \emph{we first apply $s_{i}$, then $s_{j}$}. 
\end{convention}
\begin{proposition}\label{weylelem}
    Let ${G/K}$ be a non-exceptional, compact, irreducible Hermitian symmetric space of complex dimension ${n}$. The unique Weyl group element ${w\in W_G}$ such that ${\lambda_\nu\coloneqq w(\nu\rho_{\fp}+\rho_\fk)-\rho_\fg}$ is dominant for G when $\nu\geq 1$ is ${w= {id}}$. For $-1< \nu <1 $ there is no such element. For $\nu\leq -1$, ${w}$ is given explicitly in the following table and satisfies ${\ell(w)=n}$.
    
\vspace{2mm}
    \begin{center}
    \small
\renewcommand{\arraystretch}{1.3}
\begin{NiceTabular}{|c|c|c|c|c|c|}[hlines]
\textbf{Type} & \textbf{G} & ${W_G}$ & ${w}$  \\ 
\hline 
A\txs{III} (${ p\geq q} $) & $ {\SU(p+q)} $
& ${\langle s_{\gve_i-\gve_j}\rangle}$ & ${\prod\limits_{j=1}^{q}s_{p+j-1}\cdot s_{p+j-2} \cdots s_{j}}$\\ 
BD\txs{I} ($ {p> 2 }$, odd) & $ {\SO(2+p)} $ & ${\langle s_{i}, s_{\gve_{\lceil p/2 \rceil}}\rangle}$ & ${\prod\limits_{i=1}^{\lfloor p/2 \rfloor}  s_i \cdot s_{\gve_{\lceil p/2\rceil}} \cdot \prod\limits_{i=1}^{\lfloor p/2\rfloor}  s_{\lfloor p/2\rfloor-i+1}}$ \\ 
BD\txs{I} (${ p> 2} $, even) & $ {\SO(2+p) }$ & ${\langle s_{i}, \tilde{s}_{p/2}\rangle}$ &  ${\prod\limits_{i=1}^{p/2-1}  s_i \cdot \tilde{s}_{p/2} \cdot \prod\limits_{i=1}^{p/2}  s_{p/2+1-i}}$ \\ 
D\txs{III}%, ${p}$ odd 
& $ {\SO(2p)} $& ${\langle s_{i},\tilde{s}_{p-1}\rangle}$ & ${\prod\limits_{i=1}^{ p-1}s_{p-i}\cdot s_{p-i+1}\cdot \cdots \cdot s_{p-2}\cdot\check{s}_{p,i}} $  \\ 
%D\txs{III}, ${p}$ even & $ {\SO(2p) }$ & ${\langle s_{i},\tilde{s}_{p-1}\rangle}$ & ${\prod\limits_{i=1}^{ p-1}s_{p-i}\cdot s_{p-i+1}\cdot\cdots \cdot s_{p-2}\cdot\check{s}_i} $  \\ 
C\txs{I}  & $ {\Sp(p)}  $ & ${\langle s_{i}, s_{2\gve_p}\rangle}$ & ${s_{2\gve_p} (s_{p-1} s_{2\gve_p}) \cdots  (s_1 \cdots  s_{p-1} s_{2\gve_p})}$ \\ 
%E\txs{III} &  $ E_{6}^{-78}/\Spin(10) _{Z_4}\U(1) $ & 32 \\ 
%E\txs{VII} & $ E_{7}^{-133}/\color{red}E_{6}^{-78}\color{black} _{Z_2}\U(1) $ & 54 \\
\end{NiceTabular}
\end{center}
\end{proposition}
\begin{proof}
We proceed by finding an element for which the proposition holds and use Remark \ref{uniquew} to establish its uniqueness. From Lemma \ref{rhoND}, when $\nu\geq 1$ the dominance of $\rho_{\fp}$ implies that $(\nu-1)\rho_{\fp}$ is also dominant and therefore ${w=  {id}}$. For ${-1<\nu<1}$, suppose that there exists ${w\in W_G}$ such that ${\lambda_\nu}$ is dominant. The Weyl group preserves the invariant inner product and ${\rho_{\fp}\perp \Delta_\fk}$ by Lemma \ref{Z_0*rho},
\[{\|\lambda_\nu+\rho_{\fg}\|^2=\|w(\nu\rho_{\fp}+\rho_{\fk})\|^2=\|\nu\rho_{\fp}+\rho_{\fk}\|^2=\nu^2\|\rho_{\fp}\|^2+\|\rho_{\fk}\|^2<\|\rho_{\fg}\|^2.
}\]
But, ${\lambda_\nu}$ is dominant and therefore
\[{
\|\lambda_\nu+\rho_{\fg}\|^2
=
\|\rho_{\fg}\|^2
+2\langle\lambda_\nu,\rho_{\fg}\rangle
+\|\lambda_\nu\|^2
\geq
\|\rho_{\fg}\|^2.
}\]
Hence there is no ${w\in W_G}$ for which ${\lambda_\nu}$ is dominant when ${-1<\nu<1}$. For ${\nu\leq -1}$, we prove here the case of type A\txs{III}. For the other cases we refer to \cite{Tsi26}.
        \[{\nu\rho_\fp+\rho_{\fk}=\sum_{i=1}^{p}\underset{\eqqcolon A_i}{\underbrace{\tfrac{\nu q+p+1-2i}{2}}}\gve_i+\sum_{j=1}^{q}\underset{\eqqcolon B_j}{\underbrace{\tfrac{-\nu p+q+1-2j}{2}}\gve_{p+j}}.}\]
        We note that for ${1\leq i\leq p-1}$, ${A_i-A_{i+1}=1}$, and for all ${1\leq j\leq q-1}$, ${B_j-B_{j+1}=1}$. % For $-1<\nu<1$, if such ${w}$ existed then $ {w(\nu\rho_\fp+\rho_{\fk})_1\geq(\rho_\fg)_1}$, which is a contradiction due to the following. The first equality follows from the fact that the Weyl group consists only of permutations.
        %\[{w(\nu\rho_\fp+\rho_{\fk})_1=\max\lb \tfrac{\nu q+p-1}{2},\tfrac{-\nu p+q-1}{2}\rb <\tfrac{ q+p-1}{2}=(\rho_\fg)_1.}\]
        For $\nu\leq -1$,  
        \[{B_q-A_1=\tfrac{-\nu p-q+1}{2}-\tfrac{\nu q+p-1}{2}=1-\tfrac{(\nu+1)(p+q)}{2}\geq 1.}\]
        In order for $\nu\rho_\fp+\rho_{\fk}$ to be dominant for $\fg$, even before subtracting $\rho_\fg$, we need to find an element ${w}$ such that 
        \[{w(\gve_i)=\gve_{q+i}, \quad 1\leq i\leq p, \quad  w(\gve_{p+j})=\gve_{j}, \quad 1\leq j\leq q.}\]
        So,
        \begin{align*}
            {w(\nu\rho_\fp+\rho_{\fk})-\rho_\fg}&{=\sum_{j=1}^{q}\tfrac{-\nu p+q+1-2j}{2}\gve_j+\sum_{i=1}^{p}\tfrac{\nu q+p+1-2i}{2}\gve_{q+i}-\sum_{j=1}^{p+q}\tfrac{p+q+1-2j}{2} \gve_j}\\
            &{=\sum_{j=1}^{q}\tfrac{(-\nu-1)p }{2}\gve_j+\sum_{i=1}^{p}\tfrac{(\nu+1)q }{2}\gve_{q+i}}{=(-\nu-1)\tfrac{(p+q)}{2}\varpi_q,}
        \end{align*}   
        which is dominant for $\fg$. This unique Weyl group element is expressed in terms of ${s_{i}}$ as $w=\prod\limits_{j=1}^{q}s_{p+j-1}s_{p+j-2}\cdots s_{j}$ and therefore ${\ell(w)=pq}$.
\end{proof}
\begin{corollary}\label{weylhwdim}
    Let ${G/K}$ be a compact, non-exceptional irreducible Hermitian symmetric space. The highest weight $\lambda_\nu$ of \cref{weylelem} as well as the dimension of ${\ker D(\mathcal{E}_\nu)}$ are given below.
    \vspace{2mm}
\LARGE
\newcommand{\vs}{\rule[-4ex]{0pt}{8ex}}
\newcommand{\vsb}{\rule[-8ex]{0pt}{19ex}}
\begin{center}
\renewcommand{\arraystretch}{1.6}
\setlength{\tabcolsep}{12pt}
\resizebox{\textwidth}{!}{%
\begin{tabular}{|c|c|c|c@{\hspace{1.5em}}|}
\hline
\textbf{Type} & \textbf{G} & \textbf{$\lambda_\nu$} & \textbf{$\dim(\ker D(\mathcal{E}_\nu))$}\\
\hline
A\txs{III} $(p\geq q)$ & $\SU(p+q)$ &
$\left\{\begin{array}{@{}l@{\hspace{1.5em}}l@{\hspace{1em}}}
(\nu-1)\rho_{\fp}, & \nu\geq 1\\[2ex]
-\tfrac{(\nu+1)(p+q)}{2}\varpi_{q}, & \nu\leq -1
\end{array}\right.$ &
$\vsb\left\{\begin{array}{@{}l@{\hspace{1.5em}}l@{\hspace{1em}}}
\displaystyle\prod_{i=1}^{p}\prod_{j=p+1}^{p+q}\Bigl(1+\tfrac{(\nu-1)(p+q)}{2(j-i)}\Bigr), & \nu\geq 1\\[5ex]
\displaystyle\prod_{i=1}^{q}\prod_{j=q+1}^{p+q}\Bigl(1+\tfrac{(|\nu\,|-1)(p+q)}{2(j-i)}\Bigr), & \nu\leq -1
\end{array}\right.$\\
\hline
BD\txs{I} $(p>2)$ & $\SO(2+p)$ & $(|\nu\,|-1)\rho_{\fp}$ &
$\vs|\nu\,|\,\dbinom{p-1+\tfrac{(|\nu\,|-1)p}{2}}{\tfrac{(|\nu\,|-1)p}{2}}$\\
\hline
D\txs{III}, $p$ odd & $\SO(2p)$ &
$\left\{\begin{array}{@{}l@{\hspace{1.5em}}l@{\hspace{1em}}}
(\nu-1)\rho_{\fp}, & \nu\geq 1\\[2ex]
(|\nu\,|-1)(p-1)\varpi_{p-1}, & \nu\leq -1
\end{array}\right.$ &
$\vsb\left\{\begin{array}{@{}l@{\hspace{1.5em}}l@{\hspace{1em}}}
\displaystyle\prod_{1\leq i<j\leq p}\Bigl(1+\tfrac{(\nu-1)(p-1)}{2p-(i+j)}\Bigr), & \nu\geq 1\\[5ex]
\dbinom{|\nu\,|(p-1)}{(|\nu\,|-1)(p-1)}\displaystyle\prod_{1\leq i<j\leq p-1}\Bigl(1+\tfrac{(|\nu\,|-1)(p-1)}{2p-(i+j)}\Bigr), & \nu\leq -1
\end{array}\right.$\\
\hline
D\txs{III}, $p$ even & $\SO(2p)$ & $(|\nu\,|-1)\rho_{\fp}$ &
$\vs\displaystyle\prod_{1\leq i<j\leq p}\Bigl(1+\tfrac{(|\nu\,|-1)(p-1)}{2p-(i+j)}\Bigr)$\\
\hline
C\txs{I} & $\Sp(p)$ & $(|\nu\,|-1)\rho_{\fp}$ &
$\vs\displaystyle\prod_{1\leq i<j\leq p}\Bigl(1+\tfrac{(|\nu\,|-1)(p+1)}{2p+2-(i+j)}\Bigr)\dbinom{p+\frac{(|\nu\,|-1)(p+1)}{2}}{\frac{(|\nu\,|-1)(p+1)}{2}}$\\
\hline
%E\txs{III} & $E_6(-78)$ &$\left\{\begin{array}{@{}l@{\hspace{1.5em}}l@{\hspace{1em}}}(\nu-1)\rho_{\fp}, & \nu\geq 1\\[2ex]6(|\nu\,|-1)\varpi_{1}, & \nu\leq -1\end{array}\right.$ & \\\hlineE\txs{VII} & $E_7(-133)$ & $(|\nu\,|-1)\rho_{\fp}$ & \\ \hline
\end{tabular}}
\vspace{1mm}
\end{center}
\end{corollary}
\begin{proof}
    The proof is a combination of \cref{weylelem} and the Weyl dimension formula. We refer to \cite{Tsi26} for more details.
\end{proof}
    \begin{example}\label{ex:su3realisation}
For $\fg=\mathfrak{su}(3)$, type A\txs{III}, $p=2$, $q=1$, $n=2$, $N=3$, admissible twists
are $\nu=1+2\kappa/3$, $\kappa\in\Z$. With $\varpi_{1}=\tfrac23\epsilon_{1}-\tfrac13\epsilon_{2}-\tfrac13\epsilon_{3}$
and $\rho_{\fp}=\tfrac32\varpi_{2}$,
\[
\ker D^{+}(\mathcal{E}_{\nu})=
\begin{cases}
V_{3(\nu-1)\varpi_{2}/2}, & \nu\ge1,\\
0, & \nu=\pm\tfrac13,\\
V_{-3(\nu+1)\varpi_{1}/2}, & \nu\le-1,
\end{cases}
\qquad \dim(\ker D^{+}(E_{\nu}))=\frac{9\nu^{2}-1}{8}.
\]\end{example}
\begin{proposition}\label{exeptweylelem}
    Let ${G/K}$ be an exceptional, compact, irreducible Hermitian symmetric space of complex dimension ${n}$. The unique Weyl group element ${w\in W_G}$ such that ${\lambda_\nu\coloneqq w(\nu\rho_{\fp}+\rho_\fk)-\rho_\fg}$ is dominant for G when $\nu\geq 1$ is ${w= {id}}$. For $-1< \nu <1 $ there is no such element. For $\nu\leq -1$, ${w}$ is given explicitly in the following table and satisfies ${\ell(w)=n}$. Here ${s_i\coloneqq s_{\alpha_i}}$.
    \vspace{2mm}
    \begin{center}
\renewcommand{\arraystretch}{1.3}
\small
\begin{NiceTabular}{|c|c|c|c|c|c|}[hlines]
\textbf{Type} & \textbf{G} & ${w}$ \\ 
\hline
E\txs{III} &  $ {E_{6}%^{-78}
}$ & ${s_6\cdot s_5\cdot s_4\cdot s_2\cdot s_3\cdot s_1\cdot s_4\cdot s_3\cdot s_5\cdot s_4\cdot s_2\cdot s_6\cdot s_5\cdot s_4\cdot s_3\cdot s_1 }$ \\ 
E\txs{VII} & $ {E_{7}%^{-133}
}$ & ${s_7\cdot s_6\cdot s_5\cdot s_4\cdot s_2\cdot s_3\cdot s_1\cdot s_4\cdot s_3\cdot s_5\cdot s_4\cdot s_2\cdot s_6\cdot s_5\cdot s_4\cdot s_3\cdot s_1\cdot s_7\cdot s_6\cdot s_5\cdot s_4\cdot s_2\cdot s_3\cdot s_4\cdot s_5\cdot s_6\cdot s_7 } $
\end{NiceTabular}
\end{center}
\vspace{2mm}
\end{proposition}
\begin{proof}
For $\nu\ge1$ and $-1<\nu<1$ the argument of \cref{weylelem} applies
verbatim. For E\txs{III} with $\nu\le-1$ one has
$\nu\rho_{\fp}+\rho_{\fk}=\varpi_{1}+\dots+\varpi_{5}+(6\nu-5)\varpi_{6}$, and the
element above satisfies $w(\varpi_{6})=-\varpi_{1}$ and
$w(\rho_{\fk})=-5\varpi_{1}+\varpi_{2}+\dots+\varpi_{6}$, thus $w(\nu\rho_{\fp}+\rho_{\fk})-\rho_{\fg}=6(-\nu-1)\varpi_{1}$, dominant. For E\txs{VII} we refer to \cite{Tsi26}.
\end{proof}

\begin{corollary}\label{weylhwdim2}
    The highest weight $\lambda_\nu$ of \cref{exeptweylelem} for the exceptional irreducible Hermitian symmetric spaces of type E\txs{III} is $(\nu-1)\rho_\fp$ for $\nu\geq 1$ and $6(|\nu\,|-1)\varpi_{1}$ for $\nu \leq -1$. For type E\txs{VII}, it is $(|\nu\,|-1)\rho_\fp$ for all admissible $\nu$.
\end{corollary}

\begin{corollary}\label{coro:neven}
     Let ${G/K}$ be a compact Hermitian symmetric space of complex dimension $n$ and let
    ${\lambda_\nu=w(\nu\rho_{\fp}+\rho_\fk)-\rho_\fg}$ be the highest weight of Propositions \ref{weylelem} and \ref{exeptweylelem}. For $\nu\leq-1$, ${\ell(w)=n}$ and, consequently,
    \[
        {\ker D^+(\mathcal{E}_{\nu})\neq0
        \iff n\in2\Z, \qquad\ker D^-(\mathcal{E}_{\nu})\neq0
        \iff n\in2\Z+1}.
    \]
    % Similarly for ${\ker \, D^-(\mathcal{E}_{\nu\rho_{\fp}})}$.
     For $\nu\geq1$, the Weyl element is $w=\id$, so the representation occurs in $\ker D^+(\mathcal{E}_{\nu})$ and not in $\ker D^-(\mathcal{E}_{\nu})$. For $|\nu|<1$, the kernel in question vanishes.
\end{corollary}
\begin{proof}
    The proof is a direct consequence of \cref{weylelem}, \cref{exeptweylelem} and Landweber's condition mentioned above.
\end{proof}

\subsection{Non-compact case}

When G is non-compact, discrete-series representations are classified by their Harish-Chandra parameter, while their geometric realisation is via $L^2$ kernels of Dirac operators. We have $\ker_{L^2}D\subseteq\ker_{C^\infty}D$ and every element of $\ker_{L^2}D$ is smooth by elliptic regularity.
According to  Theorem 9.3 of \cite{AS77}, ${\ker_{L^2}\,  D(\mathcal{E}_{\sigma(\lambda+\rho_\fg)-\rho_\fk})}$, where
\[{D(\mathcal{E}_{\sigma(\lambda+\rho_\fg)-\rho_\fk}):L^2(\mathcal{S\ox E}_{\sigma(\lambda+\rho_\fg)-\rho_\fk})\longrightarrow L^2(\mathcal{S\ox E}_{\sigma(\lambda+\rho_\fg)-\rho_\fk}),}\] 
    is the discrete series representation of $G$ with Harish-Chandra parameter ${\delta_\lambda}$ ${=\sigma(\lambda+\rho_\fg)}$. That holds under the assumptions that ${\lambda\in\ft_{\fg}^*}$ is such that $\langle\lambda,\check{\alpha}\rangle\in \Z$ for all ${\alpha\in \Delta_\fg}$ and $\langle\lambda,\alpha\rangle\geq 0$ for all ${\alpha\in \Delta_\fg^+}$. 
    \bigskip

    Under the first assumption, for a twist $\mu=\sigma(\lambda+\rho_\fg)-\rho_\fk$ the corresponding discrete series representation 
    \[{\ker_{L^2}\, D(\mathcal{E}_{\mu})=\ker_{L^2}\, D(\mathcal{E}_{\sigma(\lambda+\rho_\fg)-\rho_\fk})},\] 
    has Harish-Chandra parameter ${\delta_\mu=\mu+\rho_\fk}$. Therefore, for an admissible one-dimensional twist, ${\ker_{L^2} \, D(\mathcal{E}_{\nu\rho_{\fp}})}$ is the discrete series representation with Harish-Chandra parameter ${\delta_\nu=\nu\rho_{\fp}+\rho_\fk}$, if there exists a Weyl group element ${w=\sigma^{-1}}$ such that ${\lambda_\nu\coloneqq w(\nu\rho_{\fp}+\rho_\fk)-\rho_\fg\in\ft_{\fg}^*}$ satisfies
    \[{\langle\lambda_\nu,\check{\alpha}\rangle\in \Z}\text{ for all }\alpha\in \Delta_\fg\text{ and }{\langle\lambda_\nu,\alpha\rangle\geq 0}\text{ for all }\alpha\in \Delta_\fg^+.\] 
     Comparing with \cref{sec:compact}, we observe that the Weyl group elements of \cref{weylelem} are such that these conditions are satisfied. Given that ${\ell(\sigma^{-1})=\ell(\sigma)}$, the same Weyl group element determines the compact highest weight and the non-compact Harish–Chandra parameter. Their half-spinor realisations, however, require a translation of grading conventions. With the fixed exterior grading used here, the compact positive kernel requires $\ell(w)$ to be even, whereas the non-compact positive kernel requires $n+\ell(w)$ to be even, see \cite{HP06} Theorem 7.4.5 for reference. Since $\ell(w)=0$ for $\nu\geq1$ and $\ell(w)=n$ for $\nu\leq-1$, the non-compact exterior degrees are complementary to those in the compact case. We have thus established the following, which, combined with \cref{coro:neven}, results in Theorem~C of the introduction.
    \begin{corollary}\label{coro:nevennc}
        Let G/K be a non-compact, irreducible 2$n$-dimensional Hermitian symmetric space, ${E_{\nu\rho_{\fp}}}$ a one-dimensional admissible twist and ${\delta_\nu\coloneqq\nu\rho_{\fp}+\rho_\fk}$. If $\nu\geq 1$, 
        \begin{align*}
            {\ker_{L^2} \, D^+(\mathcal{E}_{\nu\rho_{\fp}})} \text{ is a discrete series representation} & {\iff  n\in 2\Z.}\\
            {\ker_{L^2} \, D^-(\mathcal{E}_{\nu\rho_{\fp}})} \text{ is a discrete series representation} & {\iff  n\in 2\Z+1.}
        \end{align*}
        In either parity, the non-zero half-spinor kernel has Harish-Chandra parameter $\delta_\nu$. If $\nu\leq -1$, the realisation is in the positive half-spinor kernel, with the same parameter. If $|\nu\,|<1$ the kernel in question vanishes.
    \end{corollary}
\section{The equivariant moduli set}

 We have seen that the kernel of the Dirac operator may be finite- or infinite-dimensional, depending on the compactness of ${G/K}$. In this section, we examine the constraints imposed by the monopole equation and their effect on the geometry of the kernel. More precisely, the quadratic monopole equations factor through a Hermitian map. After this factorisation, the equations become affine-linear conditions on a Veronese-type cone. As a consequence, the monopole constraint is represented by a section of this cone.
   \medskip
   
    For the compact case, the quadratic constraints of \cref{theo:monopole} impose a strong condition on the kernel of the Dirac operator; the equivariant moduli set, when non-empty, consists of points in the cone of both reducible and irreducible nature. In the non-compact case, the equivariant solutions sit in the ${C^\infty}$ kernel, outside the discrete series.
    \subsection{Geometry of the monopole}
    We fix a $g\in G$ and write $\Quad_{g}(\psi)$ for the evaluation of $\Quad(\psi)$ at $g$. %In the orthonormal basis of $S^{+}\otimes E_{\nu}$, $\Quad(\psi)=(\psi_i\overline{\psi_j})_{ij}$.
    %\[{\Quad(\psi)=\begin{pmatrix}{\psi_1\overline{\psi_1}}  & {\psi_1\overline{\psi_2}} & \cdots & {\psi_1\overline{\psi_{ {m^+}}}} \\\psi_2\overline{\psi_1} & \psi_2\overline{\psi_2} & \cdots & {\psi_2\overline{\psi_{ {m^+}}}} \\\vdots & \vdots & \ddots & \vdots \\{\psi_{ {m^+}}\overline{\psi_1}} & {\psi_{ {m^+}}\overline{\psi_2}} & \cdots & {\psi_{ {m^+}}\overline{\psi_{ {m^+}}}}\end{pmatrix}\eqqcolon (M_{ij})}\]
\begin{lemma}\label{Mmatrix}
    $\Quad(\psi)$ is Hermitian, positive semi-definite, of rank at most one.
\end{lemma}
\begin{proof}
    Given that ${\Quad(\psi)=\psi\psi^*}$, it is self-adjoint. It is positive semi-definite since for ${v\in \C^{2^{n-1}}}$, $v^*\psi\psi^*v$$=(\psi^* v)^*\psi^*v=|\psi^* v|^2\geq 0$. Moreover, each column is a multiple of $\psi$, thus its rank is at most one.
\end{proof}
\begin{definition}
    Let $m^+=2^{n-1}$. The following set is called the \emph{Veronese cone}
    \[{\mathcal{V}\coloneqq\{A \in \Herm(m^+) \, | \, A \text{ positive semi-definite}, \, \text{rk} A\leq 1\}}.\]
\end{definition}
By its definition, one readily verifies that the Veronese cone is a real semi-algebraic cone with a singular locus only at $0$. By \cref{Mmatrix}, we have a well defined pointwise map
\begin{align*}
   { \ker D^+ \xrightarrow{\quad \Quad_g\quad} \text{Veronese cone}\subset \Herm(m^+).}
\end{align*}
We realise the set of equivalence classes of solutions as an intersection of ${\Quad_g(\ker D^+)}$ with a real semi-algebraic subset $\mathcal{X}_n$ inside the Veronese cone, arising from the conditions of the monopole. Therefore, we define $\pi$ to be the map 
    \[{\pi:\Herm(m^+)\longrightarrow \R\op\C^{n(n-1)/2},}\]
    for which the coordinates $\pi_i(M)$ are obtained from the corresponding conditions of the monopole equation after substituting $M_{ij}=\psi_i\overline{\psi}_j$. Following \cref{monoshape}, $\pi_0(M)=\sum d_jM_{jj}$, and for $1\leq i \leq n(n-1)/2$, $ \pi_i(M)$ are given as in \cref{ex:lowdimmono}. We set $v_{n}=(\mono[n],0,\ldots,0)$ and consider
    \[\mathcal{X}_{n}(\mono):=\mathcal{V}\cap\pi^{-1}(v_{n}).\]
    \begin{example}
        According to \cref{ex:lowdimmono} for ${n=3}$, ${\pi:\Herm(4)\longrightarrow \R\op\C^{3}}$ and $v_3$ are
        \[{\pi(M)=(3M_{11}-M_{22}-M_{33}-M_{44},M_{12},M_{13}, M_{14}), \qquad v_3=(c_3,0,0,0)}\]
    \end{example}

\begin{remark}
From \cref{eqSW} and \cref{theo:monopole}, $\mathcal{S}_{\nu}^\infty(\mono)$ denotes the space of smooth solutions. Let
\[{\mathcal{Z}_{\nu,eq}^\infty(\mono)\coloneqq \left\{\psi\in\ker D^+ \,|\,\Quad_g(\psi)\in \mathcal{X}_n(\mono)\text{ for every }g\in G\right\}.}\]
It is immediate from \cref{eq:Lambdavanish} that ${\mathcal{S}_{\nu,eq}^\infty(\mono)\simeq\mathcal{Z}_{\nu,eq}^\infty(\mono)\x \{pt\}}$, where ${\{pt\}=\{\Lambda_{E_\nu}\}}$. Consequently,
\[
{\mathcal{M}_{\nu,eq}^\infty(\mono)\simeq
\mathcal{Z}_{\nu,eq}^\infty(\mono)/\mathcal{G}_{eq}.}
\]
\end{remark}
Under the condition that spinors are constant, the equivariant moduli set injects into $\mathcal{X}_n$. Injectivity is automatic by the definition of $\Quad$ and the constancy hypothesis is only needed for the injection to land in a single copy of $\Herm(m^+)$, rather than being a pointwise property.

\subsection{Equivariant moduli set and geometric interpretation}
We show that the moduli set of equivariant solutions behaves symmetrically when the complex dimension of the irreducible ${G/K}$ is even. This property does not extend to the odd-dimensional case, due to the discussion in \cref{sec:geomreal}. For that purpose, we recall Hodge--Serre duality which is realised by the induced map of the conjugate-linear Hodge star operator and Dolbeault's theorem. More specifically, if $h^\#$ denotes the conjugate-linear isomorphism induced by the Hermitian metric $h$ between $E$ and its dual, the map 
\[\overline{*}_E:\Omega^{(p,q)}(M,E)\rightarrow\Omega^{(n-p,n-q)}(M,E^*),\qquad \overline{*}_E(u\ox e)=*(\tau u)\ox h^\#(e),\]
induces a complex linear isomorphism
\[{H^{(p,q)}\,(M,\mathcal{E})\simeq H^{(n-p,n-q)}\,(M,\mathcal{E}^*)^*.}\]
Applying Dolbeault's theorem one gets the following.
\begin{proposition}[Hodge--Serre duality]\label{hodgeserre}
Let $M$ be a compact complex manifold of dimension ${n}$ and $\mathcal{E}$ be a Hermitian holomorphic vector bundle over $M$. 
\[{H^{(0,q)}\,(M,\mathcal{E})\simeq H^{(0,n-q)}\,(M,\mathcal{K}\otimes \mathcal{E}^*)^*.}\]
\end{proposition}
Cohomology groups play a crucial role in the geometric realisation of irreducible representations: we recall that Bott, generalising the theorem of Borel and Weil, in Theorem IV of \cite{Bo57}, proves that for a complex semisimple Lie group ${G}^\C$ and any parabolic subgroup ${P}$ of $G^\C$ with irreducible representation ${E_\mu}$, if there is a unique Weyl element ${w\in W_G}$ such that $w(\mu+\rho_\fg)-\rho_\fg$ is dominant for $G$,
    \begin{equation}\label{bott}
        H^{(0,q)}(G^\C/P,\mathcal{E}_\mu)=\begin{cases}
            V_{w(\mu+\rho_\fg)-\rho_\fg}, & q=\ell(w),\\
            0, & \text{otherwise}.
        \end{cases}
    \end{equation}
    If there is no such element, the cohomology vanishes for all $0\leq q\leq n$. 
    \medskip
    
    In our case, $M$ is a compact irreducible Hermitian symmetric space and $\mathcal{E}$ is a Hermitian holomorphic line bundle. Following \cref{sec:spin}, we write ${\mathcal{E}_{\mu}}$ for ${\mathcal{E}_{\mu\rho_\fp}}$, where ${\mu= 2\kappa/N}$ for some $\kappa$ as in \cref{Liftinglemma}. Consequently, its dual is $\mathcal{E}_{-\mu}$ and Hodge--Serre duality is expressed as 
\[{H^{(0,q)}\,(G/K,\mathcal{E}_\mu)\simeq H^{(0,n-q)}\,(G/K,\mathcal{E}_{-2-\mu})^*.}\] 
We take $G^\C$ to be the complexification of groups $G$ displayed in the table of \cref{subsec:HSS}. Choosing $P=K^\C \, \ltimes\exp(\fp_\C^-)$, one gets $G^\C/P\simeq G/K$.
\begin{corollary}\label{nu2nu}
    For a compact irreducible Hermitian symmetric space of even complex dimension, 
\begin{equation*}
    {\ker D^+ (\mathcal{E}_\nu)\, \simeq\, \ker D^+(\mathcal{E}_{-\nu})^*.}
\end{equation*}
\end{corollary}
\begin{proof}
    From Theorems 3 and 4 of \cite{La00}, Theorem IV of \cite{Bo57} and Hodge--Serre duality we get
\[\ker D^{+}(\mathcal{E}_{\nu})=\bigoplus_{q\ \mathrm{even}}H^{(0,q)}\bigl(G/K,\mathcal{E}_{(\nu-1)}\bigr)\simeq \bigoplus_{q\ \mathrm{even}}H^{(0,n-q)}\bigl(G/K,\mathcal{E}_{(-\nu-1)}\bigr)^*=\ker D^{+}(\mathcal{E}_{-\nu})^*.\]
\end{proof}
Following the isomorphism between the kernel and its image under $\overline{*}_E$, we next examine how the monopole behaves under $\overline{*}_E$. Since the monopole conditions apply pointwise on each fiber, we describe the action of $\overline{*}_E$ via an antiunitary involution.
\begin{definition}
    If ${u_I=X_{-\alpha_{i_1}}^*\w \cdots \w X_{-\alpha_{i_{2k}}}^*}$, ${I=\{ 1\leq i_1<\cdots<i_{2k}\leq n\}}$ is an element of ${S^+=\Lambda^{ev}{\fp_\C^-}^*}$, we denote by ${I^c}$ the complement of ${I}$. For $\gve_I\in\{\pm1\}$ the \emph{complementary operator} $L$ satisfies
\[{L\, (a_Iu_I)\coloneqq \gve_I\overline{a_I} u_{I^c}, \quad \text{ where }\quad  u_I\wedge u_{I^c}=\gve_I\, u_{top}.}\]
Its extension on $S^+\ox E_\nu$ by tensoring with $h^\#$ is denoted by the same symbol.
\end{definition}
\begin{lemma}\label{Cmap}
 
 When $n$ is even, the operator ${L}$ is an antiunitary involution on ${S^+\ox E_\nu}$ and satisfies
\begin{align*}
    {L\circ \epsilon\bigl(X_{-\alpha_i}^*\bigr)\circ \epsilon\bigl(X_{-\alpha_j}^*\bigr)}\,&={-\iota\bigl(X_{\alpha_i}^*\bigr)\circ \iota\bigl(X_{\alpha_j}^*\bigr)\circ L,}\\
    {L\circ \iota\bigl(X_{\alpha_i}^*\bigr)\circ \iota\bigl(X_{\alpha_j}^*\bigr)}\,&={-\epsilon\bigl(X_{-\alpha_i}^*\bigr)\circ \epsilon\bigl(X_{-\alpha_j}^*\bigr)\circ L.}
\end{align*}
\end{lemma}
\begin{proof} Immediate by applying the definitions.
\end{proof}
\begin{proposition}\label{XX}
Let ${n}$ be even. If ${\psi}$ satisfies the monopole conditions with parameter ${\mono}$, then ${L\psi}$ satisfies the same system with parameter ${-\mono}$.
\end{proposition}
\begin{proof}
    By \cref{selfduality} and Clifford multiplication $\gamma$ in \cref{sec:spin}, we note that
\begin{align}\label{geeii}
    \begin{split}
    {\gamma\bigl(\sqrt{2}\eta_{ij}\bigr)}\,&={2\left(\epsilon\bigl(X_{-\alpha_i}^*\bigr)\circ\epsilon\bigl(X_{-\alpha_j}^*\bigr)+\iota\bigl(X_{\alpha_i}^*\bigr)\circ\iota\bigl(X_{\alpha_j}^*\bigr)\right)}\\
    {\gamma\bigl(\sqrt{2}\theta_{ij}\bigr)}\,&={2i\left(\epsilon\bigl(X_{-\alpha_i}^*\bigr)\circ\epsilon\bigl(X_{-\alpha_j}^*\bigr)-\iota\bigl(X_{\alpha_i}^*\bigr)\circ\iota\bigl(X_{\alpha_j}^*\bigr)\right)}
    \end{split}
\end{align}
According to \cref{Cmap}, ${L}$ anti-commutes with both ${\gamma\left(\sqrt{2}\eta_{ij}\right)}$ and ${\gamma\left(\sqrt{2}\theta_{ij}\right)}$. Since ${L}$ is also unitary, the second and third families of equations are preserved by ${L}$ due to the following properties.
\[{\bigl(\gamma(\sqrt{2}\eta_{ij}) L\psi,L\psi\bigr)= -\tau(\gamma(\sqrt{2}\eta_{ij})\psi,\psi)
,\qquad \bigl(\gamma(\sqrt{2}\theta_{ij})L\psi,L\psi\bigr)=-\tau\bigl(\gamma(\sqrt{2}\theta_{ij})\psi,\psi\bigr).
}\]
Additionally, ${L}$ maps the component of degree ${2k}$ to that of degree ${n-2k}$. Therefore, the corresponding coefficient in the first equation of the monopole conditions becomes ${n-4\left(n/2-k\right)=-(n-4k)}.$ Consequently, $\sum_{k}(n-4k)|(L\psi)^{(2k)}|^{2}=-\sum_{k}(n-4k)|\psi^{(2k)}|^{2}$.
\end{proof}

\begin{theorem}[Parity pattern]\label{theo:parity}
Let ${(G/K,\lambda{B_{\fg}}_{|_{\fp}}})$ be a $2n$-dimensional irreducible compact Hermitian symmetric space, $ {E_\nu}$ be a one-dimensional admissible twist and $ {A\in \mathcal{A}^G}$. For ${n}$ odd the moduli are asymmetric in ${\mono\lessgtr0}$, while for ${n}$ even the symmetry leads to 
\[{\mathcal{M}_{\nu,eq}^\infty(\mono)\simeq\mathcal{M}_{-\nu,eq}^\infty(-\mono)}.\] 
The reducible wall is always at ${\mono=0}$, that is, ${t=-\tfrac{\nu}{2\lambda}}$.
\end{theorem}
\begin{proof}
    The asymmetry for $n$ odd is a consequence of \cref{coro:neven}. Indeed, when ${\nu\leq -1}$, one has ${\ker D^+=0}$, and therefore $\mathcal{M}_{\nu,eq}^\infty(\mono)=\varnothing$ when $\mono\neq0$; on the other hand, for $\nu=1$ the irreducible point $[(\sqrt{\mono/n}, 0,\ldots, 0), \Lambda_{E_{1}}]$, $c_n>0$, always belongs to $\mathcal{M}_{\nu,eq}^\infty(\mono)$. For ${n}$ even, the symmetry is a consequence of \cref{XX} and \cref{nu2nu}.
\end{proof}

\begin{theorem}\label{theo:compactmoduli}
    Let ${(G/K,\lambda{B_{\fg}}_{|_{\fp}}})$ be a $2n$-dimensional irreducible compact Hermitian symmetric space, $ {E_\nu}$ be a one-dimensional admissible twist and $ {A\in \mathcal{A}^G}$. The equivariant smooth solutions and moduli are given as follows.
\begin{align*}
{\mathcal{Z}_{\nu,eq}^\infty(\mono(t,\nu,\lambda))}\ &\simeq\
\begin{cases}
\text{the circle }{\{ |\psi^{(0)}|=\sqrt{\mono/n}\}\simeq \mathbb{S}^1}, & {\mono>0,\  \nu=1,}\\[2pt]
{\varnothing,}& {\mono>0,\  \nu\neq 1,}\\[2pt]
\text{the singleton } \{\psi=0\}, & {\mono=0,}\\[2pt]
{\varnothing}, &{\mono<0,\ (\nu\neq -1 \text{ or n odd} ),}\\[2pt]
    \text{the circle }{\{ |\psi^{(n)}|=\sqrt{|\mono|/n}\}\simeq \mathbb{S}^1}, &{\mono<0, \ }\nu=-1\text{, n even.} 
\end{cases}\\
{\mathcal{M}_{\nu,eq}^\infty(\mono(t,\nu,\lambda))}\ &\simeq\
\begin{cases}
{\{pt\},} & {\mono>0,\  \nu=1,}\\[2pt]
{\varnothing,} & {\mono>0,\ \nu\neq 1}\\[2pt]
{\{\text{reducible pt}\},} & {\mono=0,}\\[2pt]
{\varnothing,} &{\mono<0,\ (\nu\neq -1\text{ or n odd}),}\\[2pt]
{\{pt\},} &{\mono<0, \ \nu=-1,\ \text{n even}.}
\end{cases}
\end{align*}
\end{theorem}
\begin{proof}
According to Theorems 3 and 4 of \cite{La00} and \cref{bott}
\[\ker D^{+}(\mathcal{E}_{\nu})=\bigoplus_{q\ \mathrm{even}}H^{(0,q)}\bigl(G/K,\mathcal{E}_{(\nu-1)\rho_{\fp}}\bigr).\]
By Theorem IV of \cite{Bo57}, \cref{weylelem} and \cref{exeptweylelem}, the kernel has at most one non-zero summand, and the degree is forced to be $q\in\{0,n\}$. Suppose $\psi$ is concentrated in degree $0$ and $\mono\ne0$. By \cref{geeii}, for the $\eta_{ij}$, $\theta_{ij}$-conditions, $\gamma(\eta_{ij})$, $\gamma(\theta_{ij})$ shift exterior degree by $\pm2$; since $ \psi$ is concentrated in a single degree, the pairings vanish automatically. The first equation of
\cref{theo:monopole} reads 
\[n\,|\psi^{(0)}|^{2}=\mono, \qquad\text{so }\mono>0\text{ and }|\psi^{(0)}|\text{ is pointwise constant.}\] 
By \cref{uweights}, the degree-zero component is a section of $\mathcal{E}_{\nu-1}$, so \cref{cor:muzero}, gives $\nu=1$. The solutions form the circle of radius $\sqrt{\mono/n}$, on which $\mathcal{G}_{eq}\simeq U(1)$ acts simply transitively. Suppose now that $\psi$ is concentrated in degree $n$. This requires $n$ even, since otherwise the summand lies in $S^{-}$ and $\ker D^{+}=0$ by \cref{coro:neven}. Applying \cref{hodgeserre} and \cref{theo:parity} reduces this case to the previous one with $\mono$ replaced by $-\mono$ and $\nu$ by $-\nu$, giving $\nu=-1$, $\mono<0$ and $|\psi^{(n)}|=\sqrt{|\mono|/n}$. Finally, if $\mono=0$ then in either concentration degree the first equation forces $\psi=0$, and the invariant connection $\Lambda_{E_{\nu}}$ is unconstrained.
\end{proof}
\begin{corollary}\label{cor:injection}
The quadratic map $\Quad$ induces an injection $\mathcal{M}^{\infty}_{\nu,eq}(\mono)\hookrightarrow \mathcal{X}_{n}(\mono)$, whose image is a single point in the Veronese cone whenever the source is non-empty.
\end{corollary}
\medskip
In the four-dimensional case, for the canonical $\Spin^c$ structure $\mathfrak {s}_0$ and one distinguished value of the parameter, the assumption $A\in\mathcal{A}^G$ can be dropped: the solution of the $(-\tfrac{\nu}{2\lambda}+\tfrac12)i\omega$-perturbed equations for $\nu=1$ is unique up to
gauge equivalence and equals $\bigl((2,0),\Lambda_{E_{1}}\bigr)$, by the analytic argument of \cite{W94} in general and by a representation-theoretic argument for Hermitian locally symmetric spaces in \cite{Sp00}. %In particular for $\nu=1$ and any perturbation $(-\tfrac{\nu}{2\lambda}+s)i\omega$ with $s\ge0$, every solution is equivariant. 
Consistently with \cref{theo:compactmoduli}, this value of $t$ gives $\mono[2]=8$ and $|\psi|=\sqrt{\mono[2]/2}=2$.
\begin{theorem}\label{ZM2}
    Let ${(G/K,\lambda{B_{\fg}}_{|_{\fp}}})$ be a $2n$-dimensional irreducible non-compact Hermitian symmetric space, $ {E_\nu}$ be a one-dimensional admissible twist and $ {A\in \mathcal{A}^G}$. The equivariant solutions and moduli are given as follows.
\begin{align*}
{\mathcal{Z}_{\nu,eq}^{L^2}(\mono)}\ &\simeq\
\begin{cases}
\varnothing, & {\mono\neq 0,}\\[2pt]
\text{the singleton } \{\psi=0\}, & {\mono=0.}%\\[2pt]
%{\varnothing}, &{\mono<0.}\\[2pt]
\end{cases}    \\
{\mathcal{M}_{\nu,eq}^{L^2}(\mono)}\ &\simeq\
\begin{cases}
{\varnothing,} & \hspace{12mm}  {\mono\neq0,}\\[2pt]
{\{\text{reducible pt}\},} & \hspace{12mm} {\mono=0.}
%\\[2pt]{\varnothing,} &{\mono<0.}
\end{cases}
\end{align*}
\end{theorem}
\begin{proof}
    By Theorem 9.3 of \cite{AS77} the $L^{2}$-kernel is again concentrated in degree $0$ or $n$. If $\mono\ne0$, the first equation of \cref{theo:monopole} forces $|\psi|^{2}$ to be a non-zero constant function, hence $\|\psi\|_{L^{2}}=\infty$ on a space of infinite volume, which is a contradiction. Therefore, $\mono=0$, and then $|\psi|=0$.
\end{proof}
\begin{remark}
    In the non-compact case, there exist equivariant solutions which are smooth but not ${L^2}$. For such solutions, $\psi$ lies in ${\ker_{C^{\infty}}D^+}$ but outside the discrete-series realisation ${\ker_{L^2}D^+}$, and the compact finite-energy variational argument does not apply. 
\end{remark}

\begin{lemma}
    Let ${(G/K,\lambda{B_{\fg}}_{|_{\fp}}})$ be a $2n$-dimensional irreducible compact Hermitian symmetric space. If %${\mathcal{M}_\nu^\infty(c_n)\coloneqq}$ ${ \mathcal{S}_\nu^\infty/\mathcal{G}}$ is the moduli space and 
    ${(\mathcal{M}_\nu^\infty(c_n))^G}$ denotes the fixed-point locus of the ${G}$-action,
\begin{equation}\label{eqequalG}
    {(\mathcal{M}_\nu^\infty(c_n))^G=\mathcal{M}_{\nu,eq}^\infty(c_n)}.
\end{equation}
\end{lemma}
\begin{proof}
    If ${[\psi,A]}$ is fixed by $G$, then $F_A$ is $G$-invariant. We average $A$ over the compact group $G$. The averaged connection is $G$-invariant and has the same curvature, because $F_A$ is already invariant. Hence, the difference between the two connections is a closed $i\R$-valued one-form. Since 
    \[H^1(G/K,\R) =0\]
    for compact irreducible Hermitian symmetric spaces, this difference is removed by a gauge transformation. Thus, the gauge class has a representative with $A\in A^G$. For such a representative, fixedness of the class gives, for each $g\in G$, a gauge transformation $u_g$ with $g^*\psi = u_g\psi$. Because $A$ is $G$-invariant, \cref{prop:gaugeeq} implies that $u_g$ is constant on $G/K$. If $\psi\neq0$ the cocycle relation makes $g\mapsto u_g$ a continuous character of $G$. Since $G$ is connected and simple, this character is trivial. Hence $g^*\psi = \psi$ for all $g$. If $\psi=0$, the conclusion is immediate. The reverse inclusion is immediate from \cref{theo:compactmoduli}.
\end{proof}
We recall from \cref{SWpara} the parametrisation of effective divisors of a Kähler manifold through solutions of the vortex equations. \cref{equivreg}, which recovers Theorem~D of the introduction, is the exact equivariant analogue of \cref{Brad1}. We set $\mu\coloneq\deg_\omega(\mathcal{E}_{\nu-1})$ and consider the reparametrisation $t\mapsto nt+ \tfrac{n}{2\lambda}$, under which
\[b\Bigl(\deg_\omega(\mathcal{E}_{\nu-1}),\,nt+ \frac{n}{2\lambda}\Bigr)
=nt+ \frac{n}{2\lambda}-\frac{4\pi\deg_\omega(\mathcal{E}_{\nu-1})}{\vol(M)}=n\Bigl(t+\frac{\nu}{2\lambda}\Bigr)=\frac{\mono(t,\nu,\lambda)}{8}.\]
The two problems are not identical: the equivariant Seiberg--Witten equations carry the curvature of $\mathcal{L}_{\nu}=\mathcal{E}_{\nu}\ox\mathcal{E}_{\nu}$, whereas the vortex field is a section of $\mathcal{F}=\mathcal{E}_{\mu}$ and carries the curvature of that.
\begin{theorem}[Equivariant rigidity]\label{equivreg}
    Let $M=G/K$ be compact irreducible, let $E_\nu$ be a one-dimensional admissible twist and let $\mu= \deg_\omega(\mathcal{E}_{\nu-1})$, $\nu\geq 1$. Then
   \begin{equation}
 \left(\mathcal{M}_\nu^\infty(\mono)\right)^G\cong
 \begin{cases}
   \left(\Div_+^{\mu}(M)\right)^G,&\mono>0,\\
   \left(\Hol^{\mu}(M)\right)^G,&\mono=0,\\
   \varnothing,&\mono<0.
 \end{cases}
 \end{equation}
\end{theorem}
\begin{proof}
    On a Hermitian symmetric space there are no non-zero effective $G$-invariant divisors, so $(\Div_+^{\mu}(M))^{G}$ $=\varnothing$ for every $\mu>0$, equivalently $\nu>1$. The case $\mu=0$, equivalently $\nu=1$, describes the zero divisor which corresponds to the nowhere-vanishing section of \cref{theo:compactmoduli}. For $\mono=0$, the $G$-equivariant holomorphic line bundles are parametrised by $\nu$ as in \cref{1-1cor} and the same holds for the reducible points of \cref{theo:compactmoduli}.
\end{proof}
\begin{remark}
    
For $\nu\leq -1$ and $n$ even, \cref{theo:parity} identifies the moduli at $(c_n,\nu)$ with those at $(-c_n,-\nu)$, so the negative branch is the Serre-dual copy of the same trichotomy with $\mu=\deg_\omega(\mathcal{E}_{-\nu-1})$. For $n$ odd, it is empty by \cref{coro:neven}.

\end{remark}

\newpage

\noindent{\bf Acknowledgements.} {The results presented in this paper largely originate from the author's doctoral thesis. The author would like to thank Panagiotis Batakidis and Salah Mehdi for their unwavering support and encouragement over the years. She would also like to thank Manousos Maridakis for introducing her to Gauge Theory and for the many occasions on which she has profited from his insights. 
\\
The present research project was supported by the Hellenic Foundation for Research and Innovation (HFRI) under the “4th Call for H.F.R.I. Ph.D. Fellowships” (F.N.: 10858).}

\end{document}